\documentclass[12pt]{amsart}
\usepackage{amssymb,enumitem}
\usepackage{bbm}
\newtheorem{theorem}{Theorem}

\newtheorem{lemma}[theorem]{Lemma}
\newtheorem{corollary}[theorem]{Corollary}
\newtheorem{proposition}[theorem]{Proposition}

\theoremstyle{definition}
\newtheorem{definition}[theorem]{Definition}

\newtheorem{fact}[theorem]{Fact}
\theoremstyle{remark}

\numberwithin{equation}{section}

\newcommand{\sai} {\mbox{$\to \kern -0.50 em \to$}}
\newcommand{\nsai} {\mbox{$\not\to \kern -0.50 em \to$}}

\newcommand{\tn}{\operatorname{tn}}
\newcommand{\mv}{\operatorname{mv}}
\newcommand{\gm}{\operatorname{gm}}

\begin{document}

\title[Elastic Banach spaces are universal
\today]{Separable Elastic Banach spaces are universal}

%  Information for first author
\author{Dale E. Alspach
}
%   Address of record for the research reported here
\address{Department of Mathematics, Oklahoma State University, Stillwater
OK 74078, USA}
%\curraddr{Department of Mathematics and Statistics,}
\email{alspach@math.okstate.edu}
%\thanks{xxx}
\thanks{The second author was partially supported by Simons Foundation
Grant \#208290.}
%  Information for second author
\author{B\"unyamin Sar\i}
\address{Department of Mathematics, University of North Texas, 1155
Union Circle \#311430,
Denton, TX 76203-5017}
\email{bunyamin@unt.edu}

%    General info

%\date{January 1, 2001 and, in revised form, June 22, 2001.}

\dedicatory{To the memory of Edward Odell}


\begin{abstract} { A Banach space $X$ is {\em elastic} if there is a
constant $K$ so that whenever a Banach space $Y$ embeds into $X$, then
there is an embedding of $Y$ into $X$ with constant $K$. We prove that
$C[0,1]$ embeds into separable infinite dimensional
elastic Banach spaces, and therefore
they are universal for all separable Banach spaces. This confirms
a conjecture
of Johnson and Odell. The proof uses incremental embeddings into $X$ of
$C(K)$ spaces for countable compact $K$ of increasing complexity. To
achieve this we develop a generalization of Bourgain's basis index that
applies to unconditional sums
of Banach spaces and prove a strengthening of the weak injectivity
property of these $C(K)$ that is realized on special reproducible bases.
} \end{abstract}

\maketitle

%\section*{Introduction}

%% The correct journal style for \specialsection is all uppercase; a known
%% in amsart.cls prevents this, so input must be uppercase until it is
%\specialsection*{This is a Special Section Head}
%\specialsection*{THIS IS A SPECIAL SECTION HEAD}

%%%%%%%%%%%%%%%%%%%%%%%%%%%%%%%%%%%%%%%%%%%%%%%%%%%%%%%%%%%%%%%%%%%%%%%%
%\footnote{Here is an example of a footnote. Notice that this footnote}

\par
%%%%%%%%%%%%%%%%%%%%%%%%%%%%%%%%%%%%%%%%%%%%%%%%%%%%%%%%%%%%%%%%%%%%%%%%
\section{Introduction}

A Banach space $X$ is $K$-{\em elastic} provided that if a Banach space
$Y$ embeds into $X$ then $Y$ must $K$-embed into $X$. That is, there is
an isomorphism $T$ from $Y$ into $X$ with $$\|y\|\le \|Ty\|\le K\|y\|$$
for all $y\in Y$. $X$ is called {\em elastic} if it is $K$-elastic for
some $K<\infty$. The space $C[0,1]$ is $1$-elastic simply because it is
universal; every separable Banach space $1$-embeds into $C[0,1]$. Thus
if a separable
Banach space $X$ contains an isomorphic copy of $C[0,1]$, then $X$
is elastic. Johnson and Odell conjectured that such spaces are the only
separable elastic Banach spaces. In this paper we prove this
conjecture.

\begin{theorem}\label{mainthm}
Let $X$ be a separable infinite dimensional elastic Banach space. Then
$C[0,1]$ is isomorphic to a subspace of $X$.
\end{theorem}

Johnson and Odell introduced the notion of an elastic space in
\cite{JohnOdell} where it plays a pivotal role in their proof of the main
result that for infinite dimensional separable Banach spaces
{\em the diameter in the Banach-Mazur distance of the isomorphism
class of a space is infinite.} This is derived as an
immediate consequence of the following.

\begin{theorem}[Johnson-Odell]\label{JOthm}
If $X$ is a separable Banach space and there is a $K$ so that every
isomorph of $X$ is $K$-elastic, then $X$ is finite-dimensional.
\end{theorem}

The reason for the conjecture is that, as they noted, Theorem \ref{JOthm}
would be an immediate consequence of Theorem \ref{mainthm}. We outline
this argument below.

\begin{proof}[Proof of Theorem \ref{JOthm}] In \cite{LP} it is shown
there that there are equivalent norms that `arbitrarily
distort' the usual norm of $C[0,1]$, that is, for all $n$, there exists
an equivalent norm $|\cdot|_n$ on $C[0,1]$ so that the best embedding
constant of $C[0,1]$ with the usual norm into $(C[0,1], |\cdot|_n)$
is greater than $n$.
Let $X$ be as in the hypothesis.
For each $n,$ $X$ is isomorphic to a subspace of $C[0,1]$ with norm
$|\cdot|_n$ and is $K$-elastic with this norm.
If $X$ contains a subspace $Y$ which is isomorphic
to $C[0,1]$,
$C[0,1]$ with its
usual norm is $K$-isomorphic to a subspace of $(X, |\cdot|_n)$, and
consequently, of $(C[0,1], |\cdot|_n).$ This is a contradiction for $n$
large enough.
  \end{proof}

One of the main steps in \cite{JohnOdell} in proving Theorem~\ref{JOthm}
is to show that an elastic space must contain a nice space.

\begin{theorem}[Johnson-Odell] \label{c0} Let $X$ be elastic, separable
and infinite dimensional.
Then $c_0$ is isomorphic to a subspace of $X$.
\end{theorem}

The proof of this theorem uses Bourgain's basis index and a clever
transfinite
induction argument. Our proof of Theorem~\ref{mainthm} follows the same
general outline as their argument.
The main machinery (Proposition
\ref{c0sum}) in our proof is to show that whenever a sequence of
$C(K_n)$ spaces embed into $X$ where each $K_n$ is countable and compact,
then $\Big(\sum_{n=1}^{\infty}C(K_n)\Big)_{c_0}$ embeds into $X$. This
can be seen as a higher dimensional analogue of Theorem \ref{c0}. However,
to be able to carry out such an extension by generalizing the proof
given by Johnson and Odell one faces two rather fundamental
problems. The first is that one needs to be able to do a Bourgain basis
index argument for a basis of the $c_0$-sum. There does not seem to be a
feasible way of doing so in this setting since no basis has such a simple
homogeneous structure as the usual basis of $c_0$ itself has.
We solve this problem by not working with the basis index but
rather developing an ordinal index for unconditional sums of Banach
spaces. This is done in Section 3 and may be of independent interest.
The second major problem is that the proof of Theorem~\ref{c0} requires
embedding a countable family of incrementally renormed
spaces $Y_\alpha$ into an elastic space as spans of
blocks. Working with sums of infinite dimensional spaces requires
replacing
blocks by well positioned subspaces. To be able to `dig ourselves out
of this
hole' by patching together $Y_\alpha$ spaces of the stage $\alpha$,
one needs
copies of these spaces to be well complemented with nice projections
that have rather large kernels. This is achieved by a strengthening
of a remarkable theorem of Pelczynski that separable $C(K)$ spaces are
{\em weak
injective}  \cite{P} (See also \cite[Theorem~3.1]{R}.), and of a useful
observation due
to Lindenstrauss and Pelczynski that $C(K)$ spaces have {\em reproducible
bases} \cite{LP}.  These are addressed in Section 2. The proofs of
the main machinery (Proposition \ref{c0sum}) and Theorem \ref{mainthm}
are given in Section 4. Once Proposition \ref{c0sum} is proved, Theorem
\ref{mainthm} is easily deduced using a well known theorem of Bourgain.
(See \cite[Proposition~2.3]{O}.)

\begin{theorem}[Bourgain]\label{Bourgain} If $X$ is universal for the
class of spaces $C(K)$ where $K$ is countable compact metric, then $X$
contains an isomorphic copy of $C[0,1]$.  \end{theorem}

\section{Complementably reproducible bases of $C(\omega^\alpha)$}
In this section we give strengthenings of two important properties of
$C(K)$ spaces that are instrumental for the proof of the main result.

The first property, due to Lindenstrauss and Pelczynski
\cite[Theorem~4.3]{LP}, asserts in particular
that the canonical bases (explained below) of $C(K)$ spaces for countable
compact $K$ are {\em reproducible}. A basis $(x_n)$ is reproducible with
constant $K$ if whenever $[(x_n)]$ is isometrically embedded into a space
$X$ with a basis, one can find, for every $\epsilon>0$, a block basis
in $X$ that is $(K+\epsilon)$-equivalent to $(x_n)$. In our variant,
the connection between the isomorphism and the block basis is explicit
and realized by an infinite sequence of finite processes that lends
itself to incorporating into other constructions. In the definition we
use an infinite two-player game that is played in a Banach space $Y$
with a basis $(y_j)$ for an outcome $\mathcal O$. On turn $k$ the first
player chooses a tail subspace $[(y_j)_{j\ge m_k}]$ and the second
player picks a vector $x_k \in [(y_j)_{j\ge m_k}]$. The second player
is said to have a winning strategy for an outcome $\mathcal O$ if no
matter how the first player chooses, the sequence $(x_k)_{k=1}^\infty$
satisfies $\mathcal O$. Note that since the first player can push the
supports of $x_k$'s arbitrarily far out, the resulting sequence will be
(a tiny perturbation of) a block basis of $(y_j)$.

\begin{definition}
We say that a basis $(x_n)$ of a Banach space $X$ is {\em two-player
subsequentially $C$-reproducible} if for any  sequence of positive numbers
$(\epsilon_k)_{k \in \mathbb N}$ and isomorphic
embedding $T$ of $X$ into a Banach space $Y$ with a basis $(y_n)$,
there is a winning strategy for the second player in a two-player game
in $Y$ for picking a subsequence $(T x_{n_k})_{k=1}^{\infty}$ and blocks
$(w_k)_{k=1}^\infty$ of the basis $(y_n)$ such that
\begin{enumerate} 
\item{} $\|T x_{n_k} - w_k\|<\epsilon_k$ for
each $k\in \mathbb N,$
\item{} $(x_{n_k})$ is $C$-equivalent to $(x_n)$.
\end{enumerate}
\end{definition}

Recall that any $C(K)$ space with countable compact $K$ is isomorphic
to some $C(\alpha)$ space where the latter denotes the space $C[1,
\alpha]$ of continuous functions on a successor ordinal $\alpha+1<\omega_1$
equipped with the order topology, \cite{MS}. For a given compact metric
space $K$,
the corresponding $\alpha$ is determined as follows. Let $$K^{(1)}=\{k:
\exists k_n \in K,\, n=1,2,3,\dots, k_n\ne k_m, m\ne n, k_n \rightarrow k\}$$ be
the set of limit points of $K$. Put $K^{(\alpha+1)}=(K^{(\alpha)})^{(1)}$,
and $K^{(\beta)}=\bigcap_{\alpha<\beta} K^{(\alpha)}$ if $\beta$ is a
limit ordinal. Let $o(K)$ be the smallest ordinal such that $K^{(o(K))}$
has finite cardinality or, if $K^{(\alpha)}$ is always infinite,  let
$o(K)=\omega_1$.  If $\omega^{\alpha}\le o(K) <\omega^{\alpha+1}$,
then $C(K)$ is isomorphic to $C(\omega^{\omega^\alpha})$ \cite{BP2}.

The standard bases $(x^\alpha_n)_{n=0}^{\infty}$ of $C(\omega^\alpha)$
are described inductively. For $C(\omega)$, let $x^1_0=\mathbbm{1}_{(0,
\omega]}$ and $x^1_n=\mathbbm{1}_{\{n\}}$ for all $n<\omega$.
If the basis $(x^\gamma_n)_n$ for $C(\omega^\gamma)$ is defined, then
for each $k<\omega$ let $x^\gamma_{k,n}$ have support in
$(\omega^\gamma(k-1),\omega^\gamma k]$ and satisfy
%[
\begin{equation}\label{basis-def}
x^\gamma_{k,n}(\rho)=x^\gamma_n(\rho-\omega^\gamma (k-1))
\ {\rm for}\ \omega^\gamma(k-1)<
\rho \le
\omega^\gamma k.
\end{equation}

Let $x^{\gamma+1}_0=\mathbbm{1}_{(0, \omega^{\gamma+1}]}$, and
$(x^{\gamma+1}_j)_{j\ge 1}$ be an ordering of $\{x^\gamma_{k,n}:
n=0,1,2\ldots, k\in\mathbb N\}$ such that
\begin{equation}\label{basis order}
{\rm if}\ x^{\gamma+1}_j=x^\gamma_{k,n}\ {\rm and}\
x^{\gamma+1}_m=x^\gamma_{k,n'}\ {\rm and}\ n<n',\ {\rm then}\
j<m.  \end{equation}
That is, the order of the basis is such that
whenever the support of one function is contained in another, the
top function precedes in the order.  If $\gamma$ is a limit ordinal, we
fix a strictly increasing sequence $(\gamma_k)$ with limit $\gamma$,
and let
$x^\gamma_{k,n}$ have support in
$(\omega^{\gamma_{k-1}},\omega^{\gamma_k}]$ and satisfy
$$x_{k,n}(\rho)=x^{\gamma_k}_n(\rho-\omega^{\gamma_{k-1}}),\
\ {\rm for}\ \omega^{\gamma_{k-1}}<
\rho \le
\omega^{\gamma_k},\ 
k\in\mathbb N,  n=0,1,\ldots$$
where we set $\omega^{\gamma_0}=0$.
Then proceed analogously to define $(x^\gamma_n)_{n=0}^{\infty}$
where $x^{\gamma}_0=\mathbbm{1}_{(0,\omega^\gamma]}$.

For $C_0(\omega^\alpha)$ a standard basis is $(x^\alpha_n)_{n=1}^\alpha.$
It is not hard to see with this construction that for any $\gamma$
and
$n$ the sequence $(x^\gamma_j)_{j\in M}$ where $M=\{j:\text{supp }x^\gamma_j \subsetneq
\text{supp
}x^\gamma_n\}$ is $1$-equivalent to a standard basis of $C_0(\omega^\beta)$
for
some $\beta<\gamma.$  Also we have the following.

\begin{fact}\label{basis-eq}
Consider the set of the supports of basis functions endowed with the
partial
order of inclusion.  Let $M \subset \mathbb N$ and $\phi:M \rightarrow
\mathbb N$ be injective. Then two subsequences $(x^\alpha_i)_{i\in M}$ and
$(x^\alpha_{\phi(i)})_{i\in M}$ are $1$-equivalent if and only if $\phi$
induces an order isomorphism from $\{\text{supp }x^\alpha_i:i\in M\}$
to $\{\text{supp }x^\alpha_{\phi(i)}:i\in M\}$.
\end{fact}

Note also that the
basis is dependent on the sequence $(\gamma_k)$ and the choices in the
ordering of the $(x_{k,n})$, $k\in \mathbb N.$ In this paper we are
usually able to pass to a subsequence when needed. The following lemma
shows that these choices of $(\gamma_k)$ and ordering
are a minor technical annoyance.

\begin{lemma} \label{equiv-subseq} Suppose that $(x_n)$ and $(y_n)$
are standard bases of
$C(\omega^\alpha)$, respectively, $C_0(\omega^\alpha),$ chosen as above.
Then there
exists a subsequence of
$(x_n)$ which is $1$-equivalent to $(y_n)$ and has closed span which is
contractively complemented in $C(\omega^\alpha)$, $C_0(\omega^\alpha)$,
respectively. Moreover the subsequence can be chosen by a two-player game.
\end{lemma}

\begin{proof}
First observe that if we choose a subsequence of $(x_n)$ so that if $x_j$ is
not in the subsequence then for all $n$ such that the support of $x_n$ is
contained in the support of $x_j$, $x_n$ is not in the subsequence, then
the closed span of the subsequence is contractively complemented. Thus
in the
construction below we will choose a subsequence with this property.

The proof is by induction on the countable ordinals $\alpha \ge 1.$
Because in the case of $C(\omega^\alpha)$ the procedure for choosing a
basis requires that $x_1=y_1=\mathbbm 1_{(0,\omega^\alpha]},$ we only
need to consider the case $C_0(\omega^\alpha).$ Our inductive hypothesis
is that there is a strategy for the second player in a two player
game such that for any $\beta \le \alpha$ and subsequence $(x_n)_{n
\in M}$ which is $1$-equivalent to  standard basis of $C_0(\alpha)$
and $(y_{p_k})_{k\in \mathbb N}$ which is $1$-equivalent to a standard
basis of $C_0(\beta)$, the second player is able to choose a subsequence
$(x_{n_k})_{k\in \mathbb N}$ of $(x_n)_{n \in M}$ which is $1$-equivalent
to $(y_{p_k})_{k\in \mathbb N}$.  The game requires that at turn $k$
the first player  presents a natural number $l_k$, $m_{k-1}<l_k$ and
the second player must choose an element $n_k$ of $M$ so that $l_k <n_k$.

If $\alpha=1$, the inductive hypothesis is clearly valid. Suppose that it
holds for all $\alpha<\gamma$ and that $(x_n)_{n\in M}$ and
$(y_{p_n})_{n\in
\mathbb N}$ are subsequences of some standard bases of $C_0(\omega^\xi)$
for
some $\xi \ge \gamma,$
$1$-equivalent to standard bases of $C_0(\omega^\gamma)$ and
$C_0(\omega^\beta)$, respectively, for some $\beta \le \gamma.$ Let
$(y_{q(0,l)})_{l \in \mathbb N}$ be the subsequence of $(y_{p_n})_{n\in
\mathbb N}$ of elements of maximal support. For each $l \in \mathbb N$ let
$(y_{q(l,k)})_{k\in \mathbb N}$ be the subsequence of $(y_{p_n})_{n\in
\mathbb N}$ of all elements $y_{p_j}$ such that $p_j>q(0,l)$ and the
support of
$y_{p_j}$ is contained in the support of $y_{q(0,l)}.$

The game begins with the choice of some natural number $l(1)$ by the
first player. Note that $q(0,1)=p_1$.
$(y_{q(1,k)})_{k\in \mathbb N}$ is $1$-equivalent to a standard
basis of
$C_0(\omega^{\gamma(1)})$ for some $\gamma(1)<\beta \le \gamma.$
Because $(x_n)_{n\in
M}$ is $1$-equivalent to a standard bases of $C_0(\omega^\gamma)$,
there are
infinitely many elements of maximal support $x_{m(i)}$, $i \in \mathbb
N,$,
$m(i) \in M,$ and
for each $i$ the elements $x_n$ with $n >m(i)$ and support contained
in the
support of $x_{m(i)}$ is $1$-equivalent to a standard basis of
$C_0(\omega^{\beta(i)})$ where $\beta(i)<\gamma$ and either $\sup_i
\beta(i)
=\gamma$ or $\beta(i)+1=\gamma$ for infinitely many $i$. The second
player in
game chooses some $i_1$ and $n_1=m(i_1)$ such that $\beta(i_1)
\ge \gamma(1)$
and $m(i_1) > l(0,1)=l(1)$. Let $(x_{n})_{n \in M(1)}$, $M(1) \subset M$
be the subsequence
of elements with support strictly contained in the support of $x_{n_1}.$

For the second turn of the game player
1 chooses an integer $l(2)>n_1$, and there are two
possibilities: either $p_2=q(0,2)$ or $p_2=q(1,1).$ If $p_2=q(0,2)$, then
$(y_{q(2,k)})_{k\in \mathbb N}$ is $1$-equivalent to a
standard basis of
$C_0(\omega^{\gamma(2)})$ for some $\gamma(2)<\beta \le \gamma$.
The second player chooses $n_2=m(i_2)>l(2)$ for some $i_2$ satisfying
$\beta(i_2)
\ge \gamma(2).$ Let $(x_{n})_{n \in M(2)}$, $M(2) \subset M$,
be the subsequence
of elements with support strictly contained in the support of $x_{n_2}.$
If $p_2=p(q(1),1)$ then game $1$ is started with $(x_{n})_{n \in M(1)}$
and
$(y_{q(1,k)})_{k\in \mathbb N}$ and the integer $l(2)=l(1,1) \ge n_1.$ The
second player in game $1$ chooses $n_2 \in M(2)$ with $n_2 >l(1,1)$
by the strategy from the
inductive hypothesis for $\alpha=\beta(i_1)$ and $\beta=\gamma(1)$.

Proceeding in this fashion at each turn $k$
either $p_k=q(0,j)$ for some $j$
and the
second player picks $n_k=m(i_j)>l(k)>n_{k-1}$ with
$\beta(i_j) \ge
\gamma(j)$ or $p_k=q(j,r)$ for some, $j,r$, and the $r$ turn of the game $j$ is played
with $l(k)=l(j,r)
>n_{k-1}$ to pick $n_k \in M(j).$ It is easy to see that the resulting
sequence $(x_{n_k})_{k\in \mathbb N}$ is $1$-equivalent to $(y_{p_n})_{n
\in
\mathbb N},$ and the closed span is contractively complemented. That
completes the proof of the induction step and consequently, the lemma.

\end{proof}

The subsequence $(x^\alpha_n)_{n=1}^{\infty}$ spans the subspace
$C_0(\omega^\alpha)$ of $C(\omega^\alpha)$ of
functions vanishing at $\omega^\alpha$.
It is not hard to see that the biorthogonal
functionals $({x^\alpha_n}^*)_n$ are differences of point-mass measures,
$\delta_{\gamma(n)}-\delta_{\gamma'(n)},$ where $\gamma(n)$ is the largest
ordinal in the support of $x^\alpha_n$ and $\gamma'(n)=\gamma(k),$ where
$k$ is the largest integer strictly smaller than $n$
such that the support of $x^\alpha_k$ contains the support of
$x^\alpha_n$.
(For the top level basis elements let $\gamma'(n)= \omega^\alpha.$)
It follows that for each $m$,
$[{x^\alpha_n}^*:n \le m]$ is the span of point-mass
measures and isometric to $l^m_1,$
and therefore $(x^\alpha_n)_{n=1}^{\infty}$ is a
shrinking basis. Recall that this means that for every functional
$f$ on $C_0(\omega^\alpha)$, $\|f|_{[x^\alpha_n]_{n\ge k}}\|\to 0$
as $k\to\infty$.

The proof of Lemma \ref{equiv-subseq} shows that the sequence
$(x^\alpha_n)_{n=1}^{\infty}$ has many subsequences that are
$1$-equivalent to $(x^\alpha_n)_{n=1}^{\infty}$. In fact we can say even
more.

\begin{proposition}\label{reproducible basis} A standard basis
$(x^\alpha_n)_{n=1}^{\infty}$ of $C_0(\omega^\alpha)$ is  two-player
subsequentially $1$-reproducible.  \end{proposition}

\begin{proof}
For $\alpha=1$, $(x^\alpha_n)_{n=1}^{\infty}$ is the unit vector basis
of $c_0$. Thus the result is immediate by a sliding hump argument and
the fact that $c_0$-basis is $1$-subsymmetric (i.e., $1$-equivalent
to each of its subsequences).
Suppose that for all $\beta<\alpha,$ $(x^{\beta}_i)_i$ satisfies
the conclusion. If $\alpha=\beta+1$, then by
definition of the basis (see \ref{basis-def}) $(x^\alpha_i)_{i=1}^\infty$
consists of
the sequence of top functions $(x^\beta_{k,0})_{k=1}^\infty$ and
under each
a sequence $(x^{\beta}_{k,i})_{i=1}^\infty$ which is equivalent to a
standard basis of $C_0(\omega^\beta).$
By Lemma~\ref{equiv-subseq} any
subsequence of $(x^\alpha_i)_i$ which consist of an infinite subsequence
$(x^\beta_{k_j,0})_j$ of top functions and  $(x^{\beta}_{k_j,i})_i$
bases has a subsequence which is $1$-equivalent to $(x^\alpha_i)_i$.
Thus the second player has
a winning strategy in a block subspace game by alternating two winning
strategies in a prescribed basis order; by picking an appropriate
subsequence $(x^\beta_{k_j,0})_j$  of top functions using the strategy for
$\alpha=1$ and for each $k_j$ using the strategy from the inductive
assumption to choose a subsequence
of the corresponding $(x^{\beta}_{k_j,i})_i$ bases.

The strategy in the limit
ordinal case is similar since if $\alpha$ is the limit of the sequence
$(\gamma_k)$, then the sequence of top functions $(x^{\gamma_k}_0)_k$
is $1$-equivalent to $c_0$-basis and the basis elements below each are
equivalent to a standard basis of $C_0(\omega^{\alpha_k})$ for some
$\alpha_k<\alpha.$  Again employing the strategy for $\alpha=1$ on the top
functions and from
the inductive assumption, by following the
strategy for $\alpha_{k_j}$ for those top
functions chosen,
in the required order, and using Lemma~\ref{equiv-subseq}, the subsequence equivalent to the given standard
basis of $C_0(\omega^\alpha)$ can be produced.
\end{proof}

We will prove a much stronger statement taking advantage of {\em the
weak injectivity property} of $C(K)$ spaces due to Pelczynski \cite{P}:
If a separable Banach space contains a subspace $Y$ that is isomorphic to
a $C(K)$ space, then there is a further subspace $Z$ of $Y$ such that $Z$
is isomorphic to $C(K)$ and $Z$ is {\em complemented} in $X$. In the case
$K$ is countable, we will show that $Z$ can be realized as a subspace
spanned by a subsequence of the reproducible basis and in fact, the
second player has a winning strategy to produce such a subsequence. Thus
we introduce the following general terminology.

\begin{definition}
We say that a basis $(x_n)$ of a Banach space $X$ is
{\em two-player $D$-complementably subsequentially $C$-reproducible}
if the second player has a winning strategy in the following modified
two-player game. Let $(\epsilon_k)_{k \in \{0\}\cup \mathbb N}$
be a sequence of
positive numbers, $T$ be an isomorphic embedding of $X$
into a Banach space $Y$ with a basis $(y_i).$ Suppose that $l_0=0$,
$G_0=\emptyset,$ $l_0<l_1<\dots<l_{k-1}$, a finite index set $G_{k-1},$
positive numbers $(\delta_n^{k-1})_{n\in G_{k-1}},$ and
$(v_n^{k-1})_{n\in G_{k-1}}\subset Y^*$ have been chosen.
On the $k$th turn, the first player 
chooses an integer $i_k,$ a finite number of
elements $(u_n^k)_{n\in F_k}$ of 
$Y^*$, positive real
numbers $(\rho_n^k)_{n\in F_k}$ and also chooses a finite set of blocks
$(b_j)_{j\in J_k}\subset [y_i: l_{k-1}<i\le i_k]$ 
satisfying  $|v_l^{k-1}(b)|<\delta_l^{k-1}\|b\|$ for
all $b \in [b_j:j \in J_k]$ and $l \in G_{k-1}.$ The second player
chooses an integer $l_k>i_{k},$ a finite set $M_k \subset \mathbb N$ with
$M_{k-1}<M_k,$ i.e., $\max M_{k-1}<\min M_k,$ a finite set of blocks $(w_m^k)_{m\in M_k}
\subset [y_i:i_k<i \le l_k]$ with 
$\sum_{m\in M_k} \|T x_{m}-w_m^k\|<\epsilon_{k}$, 
and with $|u_n^k(x)|<\rho_n^k\|x\|$ for all $x \in [T x_{m}:m\in M_k]$
and
$n \in F_k,$
and chooses a finite number of
elements $(v_n^{k})_{n\in G_{k}}$ of
$Y^*$ and positive real
numbers $(\delta_n^{k})_{n\in G_{k}}$. The second player wins if for
$M= \bigcup_{k=1}^\infty M_k,$
\begin{enumerate}
\item{}
$(x_m)_{m\in M}$ is $C$-equivalent to the basis $(x_n)_{n=1}^\infty,$
%$\sum_{m \in M} \|T x_m -w_m\|<\epsilon$
\item{} there is a projection $P$ of norm at most
$\|T\|\|T^{-1}\|D$ from $[\{b_j:j \in J_k, k\in \mathbb N\}\cup \{T x_m:
m\in
M\}]$ onto $[T x_m: m\in M]$ with $\|Pz\| \le \epsilon_0 \|z\|$ for all
$z \in[b_j:j\in J_k,k\in \mathbb N].$
\end{enumerate}
\end{definition}

For $p$ in a countable compact space $K$, $\delta_p$ denotes the Dirac
evaluation functional on $C(K)$. (We use $\delta_k^n$ for small positive
numbers below but the indices make the distinction clear.) 
Note that $\delta_{\omega^\alpha}=0$
for the case $C_0(\omega^\alpha)$ below. The next lemma is a recasting of
the core of Pelczynski's result for the countable case \cite{P}.

\begin{lemma}\label{weakstarHB}
Let $\alpha<\omega_1$, and let $S$ be an isomorphic  embedding of
$C(\omega^\alpha)$ (respectively, of $C_0(\omega^\alpha)$) into a separable
Banach
space $Y$, let $(x_n^\alpha)$ be a standard basis of $C(\omega^\alpha)$
(respectively, of $C_0(\omega^\alpha)$), and let $(y^*_\rho)_{\rho \le
\omega^\alpha}\subset 2\|(S^*)^{-1}\|B_{Y^*}$
satisfy $S^*y^*_\rho=\delta_\rho$ for all $\rho \le \omega^\alpha$. Then
there is a compact subset $\Gamma$ of $[1,\omega^\alpha]$ homeomorphic to
$[1,\omega^\alpha]$ and a (weak${}^*$) compact subset $(w^*_\rho)_{\rho
\in \Gamma }$ of $Y^*$ such that

(i) $S^*w^*_\rho=\delta_\rho$ for
all $\rho \in \Gamma $, for each isolated point $\gamma$ of $\Gamma$
is an isolated point of $[1,\omega^\alpha]$, $w^*_\gamma=y^*_\gamma$, and
the map $\rho \rightarrow w^*_\rho$ is a homeomorphism, 

(ii) there is a subsequence of $(x_n^\alpha)$ equivalent to $(x_n^\alpha)$,
with contractively complemented closed linear span such that
the restriction to $\Gamma$ induces an isomorphism $R$ from the span
of the subsequence onto $C(\Gamma)$ (respectively, $C_0(\Gamma)$),
and $R^*\delta_\rho=S^*w^*_\rho$ for all $\rho \in \Gamma.$
\end{lemma}

\begin{proof} The proof is by induction on $\alpha$. Suppose
$\alpha=1$. By
passing to a subsequence $(y^*_k)_{k\in K}$,
we have that $y^*_k \rightarrow y^*$ in the weak${}^*$ topology for some
$y^* \in Y^*.$ Let $\Gamma= K\cup \{\omega\}$ and put
$w^*_\omega=y^*$. The desired subsequence of the
basis $(\mathbbm{1}_{n})_{n\in \mathbb N}$ is $(\mathbbm{1}_{k})_{k\in
K}$ in the case
$C_0(\omega),$
and $x_0^1=\mathbbm{1}_{[1,\omega]}$
followed by $(\mathbbm{1}_{k})_{k\in K}$ for the case $C(\omega).$
All of the
requirements are clearly satisfied.

Now suppose that the lemma holds for all $\beta<\alpha.$ Because $(x_n^\alpha)_{n=1}^\infty$ is
a standard basis of $C_0(\omega^\alpha),$ the supports $A_k$, $k \in K,$ of the top level
elements in the basis, $(x_k^\alpha)_{k\in K}$, are intervals homeomorphic
to $[1,\gamma_k]$ where $(\gamma_k)$ is of one of two types. If
$\alpha=\beta+1$ for some ordinal $\beta$, $\gamma_k=\omega^\beta $, else
there exist $\alpha_k \nearrow \alpha$ and $\gamma_k=\omega^{\alpha_k}$
for all $k$. 
%Let $\gamma_0=0.$ 
%For each $n$, $[\gamma_{n-1}+1,\gamma_n]$
%is homeomorphic to %$[1,\omega^\zeta]$ for some $\zeta<\alpha.$
For each $k$, the elements of the basis which are supported strictly inside
$A_k$, $(x_{k,j}),$ are a standard basis of $C_0(\gamma_k)$. 
By induction, for each $k$ we can find subsets $\Gamma_k$ of
$A_k$ and $(w^*_\rho)_{\rho \in \Gamma_k}$ and
subsequences $(x_{k,j})_{j\in J_k}$ of the bases of $C_0(A_k)$
as in the conclusion. For each $k$ let $\zeta_k$ denote the highest order
point in $\Gamma_k$, i.e., $\Gamma_k^{(\xi)}=\{\zeta_k\}$ for some
$\xi<\omega_1.$ By passing to a subsequence  we may assume that
$(w^*_{\zeta_k})_{k \in K}$ converges to some $y^*.$
For each $k$ there is a w${}^*$- neighborhood
$\mathcal N_k$ of $w^*_{\zeta_k}$ such that if $y^*_k \in \mathcal
N_k$, then $(y^*_k)_{k \in K}$ converges to $y^*.$
For each
$k$ by replacing $\Gamma_k$  by a slightly smaller set and $(x_{k,j})_{j\in
J_k}$ by a corresponding
subsequence $(x_{k,j})_{j\in J'_k}$ we may assume
that $w^*_\rho\in \mathcal N_k$ for all $\rho \in \Gamma_k.$ 
Let $\Gamma=\{\omega^\alpha\}\cup \bigcup_{k\in K}
\Gamma_k$ and $w^*_{\omega^\alpha}=y^*.$ In the case of $C(\omega^\alpha)$
a subsequence of the basis equivalent to a standard basis
is $\mathbbm{1}_{[1,\omega^\alpha]}$
followed by the elements of the bases
$(x_k)_{k\in K},$ $(x_{k,j})_{j \in J_k}$, $k\in K$
in the required order. Some further thinning using Lemma
\ref{equiv-subseq} may be required to get a subsequence equivalent to
$(x_n)$. In the 
case $C_0(\omega^\alpha)$ we omit $\mathbbm{1}_{[1,\omega^\alpha]}$.
The required properties
are easily verified.  
\end{proof}

%Now we are ready to show that the spaces $C_0(\omega^\alpha)$ have
%two-player
%complementably subsequentially reproducible bases.

\begin{proposition} \label{calphaCSR} Let $\alpha$ be a countable
ordinal. A standard basis $(x^\alpha_n)_{n=1}^{\infty}$ of
$C_0(\omega^\alpha)$ is two-player $2$-complementably subsequentially
$1$-reproducible.
\end{proposition}

\begin{proof}
We will use induction on $\alpha$ with an inductive hypothesis which will
be described after the first step is proved.
Let $T$ be an isomorphism of
$C_0(\omega^\alpha)$ into a Banach space $Y$ with a bimonotone basis
$(y_n)_{n=1}^\infty$.

Assume $\alpha=1$ and let $(\epsilon_k)_{k=0}^\infty$ be a sequence of
positive numbers and $\epsilon'_0=\epsilon_0/(2\|T\|).$
For each $n \in \mathbb N,$ let $w_n^*$ be a Hahn-Banach
extension of $(T^{-1})^* \delta_n$ from $T\big(C_0(\omega^\alpha)\big)$
to $Y$. By passing to a subsequence and restricting to an isometric
subspace as given by Lemma~\ref{weakstarHB}, we may assume that $(w_n^*)$
converges weak${}^*$ to some $w^* \in Y^*.$ Note that $w^*$ is $0$
on $T(C_0(\omega^\alpha)).$ Let
$d_n^*=w_n^*-w^*$ for each $n$. For use in later steps of the induction we
will not make full use of the fact that $d_0^*= \text{w}^*\lim d_n^* =0.$
Instead we will use that it is small on certain elements. More precisely, we
put $\delta_0^0=\epsilon'_0/2^2$ and 
%will
consider $d_0^*$ as an element chosen before the first player's first
turn which imposes the condition on that player's choice of $(b_j)_{j\in
J_1},$
%\newline 
$|v_0^0(b)|=|d_0^*(b)|< \delta_0^0
\|b\|=(\epsilon'_0/2^2)\|b\|$ for all $b\in[b_j:j\in J_1].$
Because $(x_n)$
is a two-player subsequentially $1$-reproducible basis, we can
use the strategy
for that game and impose additional requirements.
We now describe the second player's move in the game.

Suppose that at step $k$ the
first player has chosen an integer $i_k$, a finite set of blocks
$(b_j)_{j\in J_k}\subset [y_n:n \le i_k]$ of the basis $(y_n)$ satisfying 
$|v_l^{k-1}(b)|<\delta_l^{k-1}\|b\|$ for
all $b \in [b_j:j \in J_k],l \in G_{k-1},$ and has chosen elements $(u_n^k)_{n\in F_k}$ 
of $Y^*$ and positive real
numbers $(\rho_n^k)_{n \in F_k}$. Because $(d^*_n)$ converges
to $d_0^*,$ $|d_0^*(b)|<\delta_0^{l-1}\|b\|$ for all $b \in
[b_j:j\in J_l]$, $l \le k,$
and $(T x_n)$ is weakly null, there exists $m'_k$ so that if
$m \ge m'_k,$
$|d_m^*(b)|<|d_0^*(b)|+(\delta_0^{l-1}/2)\|b\|$ for all  $b \in [b_j:j\in J_l]$,
$l \le k,$
and $|u_n^k(T x_m)|<\rho_n^k\|T x_m\|$ for all $n \in F_k .$
By the strategy for the sequential reproducibility game
the second player chooses  
$m_k \ge m'_k$ and $l_k$ such that there
is some block $w_{m_k}\in [y_n:i_k<n\le l_k]$ 
such that 
$\|T x_{m_k} -w_{m_k}\|<\epsilon_k.$ The second player sets
$(v_l^{k})_{l\in G_k}=(d^*_{m_l})_{l=0}^k$  and
$(\delta_l^{k})_{l\in G_k}=(\epsilon'_0 /2^{2(k+1)})_{l=0}^k$ for the
conditions on the first player's next turn.

Because the second player uses the strategy from the reproducibility game
$(x_{m_k})_{k\in \mathbb N}$ is $1$-equivalent to $(x_n)$.
Let $P z =\sum_{k=1}^\infty
d^*_{m_k}(z) T x_{m_k}$ for all $z \in Y.$ Observe that for all $k,k',$ 
$d^*_{m_k}(T x_{m_{k'}})= \delta_{m_k}(\mathbbm{1}_{m_{k'}})$. Hence $P T x_{m_k} = T x_{m_k}.$
If $z= \sum_{k=1}^K \sum_{j \in J_k} a_{k,j} b_j^k$, then
\begin{multline*}
|d^*_{m_{k'}}(z)|= |d^*_{m_{k'}}(\sum_{k=1}^{k'} \sum_{j \in J_k} a_{k,j} b_j^k)+
d^*_{m_{k'}}(\sum_{k=k'+1}^K \sum_{j \in J_k} a_{k,j} b_j^k)| 
\\
\le\sum_{k=1}^{k'} \Big(|d_0^*(\sum_{j \in J_k}
a_{k,j}b_j^k)|+(\delta_0^{k-1}/2)\|\sum_{j \in J_k} a_{k,j}
b_j^k\|\Big)+\sum_{k=k'+1}^K
\delta^{k-1}_{0}\|\sum_{j \in J_k} a_{k,j} b_j^k\|
\\ \le |d_0^*(z)|+\sum_{k=1}^{k'} (\epsilon_0'/2^{2k})\|z\|+
\sum_{k=k'+1}^K
\epsilon'_0/2^{2k}\|z\|<
%k' \epsilon'_0/2^{2k'}\|b\|+\epsilon'_0/(3\cdot 2^{2k'})<
|d_0^*(z)|+(\epsilon'_0 )\|z\|.
\end{multline*}
This completes the first step of the induction.

For the induction hypothesis we
actually want more than the statement of the proposition. This is because
for $\alpha>1$, the projection
formula (from the weak injectivity property of Pelczynski) is more
involved. Namely, we
require that
we are able to choose pairs of elements, $x_m$ from the basis and $d^*_m \in
Y^*,$ so that $T^*
d^*_m=\delta_{\gamma(m)}-\delta_{\omega^\alpha}=\delta_{\gamma(m)}$, the natural mapping $S:[x_m:m \in
M]\rightarrow C_0(\Gamma)$ where $\Gamma =\{\gamma(m):m\in M\}$ satisfying $(S
x_m)(\gamma(k)) = x_m(\gamma(k))$ is a surjective isometry, and the
projection $P$ is of the form $T E V$ where $V: Y \rightarrow C_0(\Gamma)$
is defined by $(V z) (\gamma(m))=d^*_m(z)$ for all $z\in Y,$
$E$ is the extension operator
which maps $C_0(\Gamma)$ into $C_0(\omega^\alpha)$ with range in $[x_m:m\in
M]$ with $S E= I$. Explicitly 
$$E f = \sum_{m\in M} (f(\gamma(m))-\sum_{\substack{\{m'\in
M: \\ m'\ne m, \\ x_{m'}\ge x_m\}}}f(\gamma(m')))x_m$$
or equivalently,
$$
E_f(\beta)=\begin{cases} f(\beta) \qquad \text{if } \beta \in \Gamma,\\
f(\gamma(m)) \qquad \text{if } x_m(\beta)=1, x_{m'}(\beta)=0 \text{ for
all }m'>m,\\
0 \qquad \text{else.}
\end{cases}
$$
Notice that the norm of the projection $P$ is at most $\|T\|\sup_{m\in M}
\|d_m^*\|.$

Now suppose that the following induction hypothesis
holds for all $\beta<\alpha$ and $\alpha
>1.$
 
For all sequences of positive numbers
$(\epsilon_k)_{k=0}^\infty$, maps $T:C_0(\omega^\beta)\rightarrow Y$, standard
bases $(x_n)_{n\in \mathbb N}$, with extensions $(d_n^*)_{n\in \{0\}\cup
\mathbb N}$ of\newline
$((T^{-1})^*\delta_{\gamma(n)})_{n\in \{0\}\cup \mathbb N},$ as in
Lemma~\ref{weakstarHB}, there is a winning strategy for the second player in
the complementably sequentially reproducible basis game to produce
$(x_m)_{m\in M}$ and $(d_m^*)_{m\in M}$ as in the game which also satisfy
these properties.
\begin{enumerate}
\item{} The elements of $M$ are chosen one by one, i.e., at each of the
second player's turns only one $m$ and $w_m$ are chosen with $\|T x_m -w_m\|$
as small as desired. (The required bound on the perturbation does not need
to be known until the step at which $m$ and $w_m$ are chosen.)
\item{} For each $k \in \{0\}\cup \mathbb N,$
$(v_n^k)_{n\in G_k}=(d_m^*)_{m\in M_k}$ where $M_k$ is the set with elements $\{0\}$
and the first $k$ elements of $M$.

\item{} The projection is of the form described above.
\item{} For all $z \in [b_j:j\in \cup_k J_k],$
$|d_m^*(z)|<|d_0^*(z)|+(\epsilon_0/(2\|T\|))\|z\|.$
\end{enumerate}
Here $d_0^*$ is an extension of $(T^{-1})^*\delta_{\omega^\beta}.$

Let $T:C_0(\omega^\alpha)\rightarrow Y$, $(x_n)_{n\in \mathbb N}$ be a
standard basis of $C_0(\omega^\alpha)$ with
extensions $(f_n^*)_{n\in \{0\}\cup
\mathbb N}$ of
$((T^{-1})^*\delta_{\gamma(n)})_{n\in \{0\}\cup \mathbb N}.$
First we apply Lemma~\ref{weakstarHB} to replace the original map $T$
by its restriction to a subspace isometric to $C_0(\omega^\alpha)$ spanned
by a subsequence of the basis equivalent to it such
that for each $\gamma \le \omega^\alpha$, we have $w^*_\gamma,$ a Hahn-Banach
extension of 
$(T^{-1})^*(\delta_\gamma),$ such that $\delta_\gamma \rightarrow w^*_\gamma$ is a
w${}^*$-homeomorphism. In this way we may assume that the original $T$ has these
properties.

The given basis of $C_0(\omega^\alpha)$ contains a
subsequence $(x_n)_{n\in N_0}$ such that for every $k \notin N_0$, there is
some $n \in N_0$ with $x_k \le x_n$ (pointwise) and if $n,m \in N_0$, then
$x_n \cdot x_m =0.$ The support of each $x_n$, $n \in N_0$, is homeomorphic
to $[1,\omega^{\zeta_n}]$ for some $\zeta_n <\alpha.$  To simplify the
situation choose a sequence $(n(i))_{i=1}^\infty$ in $N_0$ of distinct
elements such that
$\zeta_{n(i)} \nearrow \alpha$, if $\alpha$ is a limit ordinal,  and
$\zeta_{n(i)}=\beta$, if $\beta+1=\alpha$ for some ordinal $\beta.$ We may
discard all $x_n$ such that $x_n \cdot x_{n(i)} =0$ for all $i$. Because we
can apply Lemma~\ref{equiv-subseq} at the end of the argument to
obtain a subsequence equivalent to the original standard basis of
$C_0(\omega^\alpha)$, it will be sufficient to produce a subsequence of the
basis equivalent to {\it some} standard basis of $C_0(\omega^\alpha)$ that
satisfies all of the other requirements.

For each $i \in \mathbb N$ let $N_i=\{m: x_m \le x_{n(i)}, m\ne n(i)\}.$
$(x_m)_{m \in N_i}$ is a standard basis for $C_0(\omega^{\gamma_{n(i)}})$
for some $\gamma_{n(i)}<\alpha,$ and $(d_m^*)_{m \in \{n(i)\}\cup N_i}=(w_{\gamma(m)}^*-w_{\omega^\alpha}^*)_{m \in \{n(i)\}\cup N_i}$ is a corresponding sequence of
extensions of the inverse images of the Dirac measures. Thus for
each $i$ the induction hypothesis applies. $(x_m)_{m \in N_0}$ is a
standard basis
for $C_0(\omega)$ and $(d_m^*)_{m \in \{0\}\cup
N_0}=(w_{\gamma(m)}^*-w_{\omega^\alpha}^*)_{m \in \{0\}\cup N_0},$
where $\gamma(0)=\omega^\alpha,$ is a
corresponding sequence of
extensions of the inverse images of the Dirac measures, Therefore
there is a winning strategy as in the case $\beta=1$.
The remainder of the argument is interweaving all of the strategies to
produce the required strategy for $\alpha.$ Below at each step of the
induction there will be a finite
but increasing number of games employed so we will number the games as they
arise as game $0, 1, 2, \dots.$ We will include an extra subscript when
needed to indicate parameters associated with a particular game. For
example $\delta_{j,n}^{k}$ would be associated with game $j$. Parameters
for the combined game $\alpha$  will have a single subscript.

Let the sequence of positive numbers
$(\epsilon_k)_{k=0}^\infty$ be given and let $\delta_0^0=\epsilon_0/(4\|T\|)$.
We may assume that $(\epsilon_k)$ is non-increasing.
The first player chooses an integer $i_1,$ a finite set of blocks
$(b_j)_{j\in J_1}$ in $[y_k: k\le i_1],$ such that
$|d_0^*(b)|<\delta_0^0\|b\|$ for all $b \in [b_j:j\in J_1],$
$(u_n^1)_{n\in F_1}$ from the
dual of $Y$ and positive real
numbers $(\rho_n^1)_{n\in F_1}$. The second player views this as the first
move of game 0 and uses the strategy for
$\beta=1$ with $\epsilon_{0,0}=\epsilon_0/2^2$
to choose $l_{0,1}>i_1$, $m_{0,1} \in N_0$, and a
block $w_{m_{0,1}}\in[y_i:i_1<i\le l_{0,1}]$ such that
$\|Tx_{m_{0,1}}-w_{m_{0,1}}\|<\epsilon_{0,1}=\epsilon_1$
and $|u_n^1(Tx_{m_{0,1}})|<\rho_n^1\|Tx_{m_{0,1}}\|$ for all $n \in F_1.$
%$|d^*_{m_{0,1}}(b)|<(\epsilon_0/2^2)\|b\|$ for all $b \in [b_k:k\in K_1],$ $\|T
%x_{m_{0,1}}-w_{0,1}\|<\epsilon_0/2^2,$ and $|z_n^1(x_{m_{0,1}})|<\rho_n^1$,
%$1 \le n \le n_1.$
Then the second player chooses functionals $(v_{0,n}^1)_{n\in G_{0,1}}$ and
positive numbers $(\delta_{0,n}^1)_{n\in G_{0,1}}$ for the first player's next
turn of game 0. We set  $l_1=l_{0,1},$ $ m_1=m_{0,1},$ $ w_{m_1}=w_{m_{0,1}},$ $
(v_n^1)_{n\in G_1}=(v_{0,n}^1)_{n\in G_{0,1}},$
and $(\delta_n^1)_{n\in G_1}=(\delta_{0,n}^1)_{n\in G_{0,1}}$ for game
$\alpha.$

The first player chooses an integer $i_2>l_1,$
$(u_n^2)_{n\in F_2}$ from $Y^*,$ positive real
numbers $(\rho_n^2)_{n\in F_2},$
and blocks $(b_j)_{j\in J_2}\subset [y_i:l_1<i\le i_2]$ satisfying
$|v_n^{1}(b)|<\delta_n^1
\|b\|$ for all $b \in [b_j:j\in J_2],$ $n\in G_1.$
The second player uses the next move of the strategy for 
$\beta=1$ to choose $l_{0,2}>i_2$, $m_{0,2} \in N_0$, and a
block $w_{m_{0,2}}\in[y_i:i_2<i\le l_{0,2}]$ such that
$\|Tx_{m_{0,2}}-w_{m_{0,2}}\|<\epsilon_{0,2}=\epsilon_2$
and $|u_n^2(Tx_{m_{0,2}})|<\rho_n^2\|Tx_{m_{0,2}}\|$ for all $n \in F_2.$
Then the second player chooses functionals $(v_{0,n}^2)_{n\in G_{0,2}}$ and
positive numbers $(\delta_{0,n}^2)_{n\in G_{0,2}}$ for the first player's
next turn of game 0. We set  $l_2=l_{0,2},$ $ m_2=m_{0,2},$ $
w_{m_2}=w_{m_{0,2}},$ $
(v_n^2)_{n\in G_2}=(v_{0,n}^2)_{n\in G_{0,2}},$
and $(\delta_n^2)_{n\in G_2}=(\delta_{0,n}^2/2)_{n\in G_{0,2}}$ for game
$\alpha.$

For the third turn of the game $\alpha,$ the first player
chooses an integer $i_3>l_2,$
$(u_n^3)_{n\in F_3}$ from $Y^*,$ positive real
numbers $(\rho_n^3)_{n\in F_3},$
and blocks $(b_j)_{j\in J_3}\subset [y_i:l_2<i\le i_3]$ satisfying
$|v_n^{2}(b)|<\delta_n^2
\|b\|$ for all $b \in [b_j:j\in J_3],$ $n\in G_2.$ This time the second
player considers this his first move of game 1 with
$\epsilon_{1,0}=\epsilon_0/2^4$, the basis $(x_n)_{n\in
N'_1}$ where $N'_1=N_{m_{0,1}}$ and $(d_{1,m}^*)_{m=0}^\infty=(d_m^*)_{m\in \{m_{0,1}\}\cup
N'_1}=(w_m^*-w_{\omega^\alpha}^*)_{m\in \{m_{0,1}\}\cup
N'_1},$ ($d_{1,0}^*=d_{m_{0,1}}^*$) and uses the
strategy for $\beta_1=\zeta_{m_{0,1}}$ with conditions $i_{1,1}=i_3,$
$(u_{1,n}^1)_{n\in F_{1,1}}
=(u_n^3)_{n\in F_3}$,
$(\rho_{1,n}^1)_{n\in F_{1,1}}=(\rho_{n}^1)_{n\in F_3},$ and
$(b_j)_{j\in J_{1,1}}=(b_j)_{j\in J_1 \cup J_2 \cup J_3}.$
The second player chooses $l_{1,1}>i_3$, $m_{1,1} \in N'_1$, and a
block $w_{m_{1,1}}\in[y_i:i_3<i\le l_{1,1}]$ such that
$\|Tx_{m_{1,1}}-w_{m_{1,1}}\|<\epsilon_{1,1}=\epsilon_3$
and $|u_n^3(Tx_{m_{1,1}})|<\rho_n^3\|Tx_{m_{1,1}}\|$ for all $n \in F_3.$
The second player also chooses functionals $(v_{1,n}^1)_{n\in G_{1,1}}$ and
positive numbers $(\delta_{1,n}^1)_{n\in G_{1,1}}$ for the first player's
next
turn of game 1. We set  $l_3=l_{1,1},$ $ m_3=m_{1,1},$ $
w_{m_3}=w_{m_{1,1}},$ $
(v_n^3)_{n\in G_3}=(v_{1,n}^1)_{n\in G_{1,1}}\uplus (v_{0,n}^2)_{n\in
G_{0,2}}$
and $(\delta_n^3)_{n\in G_3}=(\delta_{1,n}^1)_{n\in G_{1,1}}\uplus
(\delta_{0,n}^2)_{n\in
G_{0,2}}$ for game
$\alpha,$ where $\uplus$ denotes concatenation of the finite sequences.

Now we briefly describe how to continue. The key point is to include in
the conditions for each move the
conditions imposed by all of the moves of the games
in progress. In order to write this more precisely we need to introduce
some notation. The moves are taken in Cantor order
%$\{0\}\cup \mathbb N \times \mathbb N,$
$(0,1),(0,2),(1,1),(0,3),\dots $ for the elements of $(\{0\}\cup \mathbb N)
\times \mathbb N.$
The turn $k$ of the game $\alpha$ is considered by the
second player to be the move $\mv(k)$ of game $\gm(k)$ where
\begin{multline*}k=\tn(\gm(k),\mv(k))=1+\gm(k)+\sum_{r=0}^{\substack{\mv(k)+\\\gm(k)-1}}
r, \\ \gm(k) \ge 0,
\text{ and }\mv(k) \ge 1 .
\end{multline*}
The first player makes the move for turn 
$k$ by choosing an integer $i_k,$ a finite set of blocks
$(b_j)_{j\in J_k}$ in $[y_i: l_{k-1}<i\le i_k],$ such that
$|v_n^{k-1}(b)|<\delta_n^{k-1}\|b\|$ for all $b \in [b_j:j\in J_k],$ $n\in
G_{k-1},$
$(u_n^k)_{n\in F_k}$ from the
dual of $Y$ and positive real
numbers $(\rho_n^k)_{n\in F_k}.$ Let $k'=\gm(k)$ and $k''=\mv(k).$
The second
player considers this his move $k''$ of game $k'$ with
$\epsilon_{k',0}=\epsilon_0/(2^{2k'+1}\|T\|)$, the basis $(x_n)_{n\in
N'_{k'}}$ where $N'_{k'}=N_{m_{0,k'}}$ and $(d_{k',m}^*)_{m=0}^\infty=(d_m^*)_{m\in
\{m_{0,k'}\}\cup
N'_{k'}}=(w_m^*-w_{\omega^\alpha}^*)_{m\in \{m_{0,k'}\}\cup
N'_{k'}},$ ($d_{k',0}^*=d_{m_{0,k'}}^*$) and uses the
strategy for $\beta_{k'}=\zeta_{m_{0,k'}}$ with conditions $i_{k',k''}=i_k,$
$(u_{k',n}^k)_{n\in F_{k',k''}}
=(u_n^k)_{n\in F_k}$,
$(\rho_{k',n}^k)_{n\in F_{k',k''}}=(\rho_{n}^k)_{n\in F_k},$ and
$(b_j)_{j\in J_{k',k''}}=(b_j)_{j\in J_{k-k'-k''+1} \cup
J_{k-k'-k''+2}\cup \dots \cup J_k},$ if $k''>1,$ $(b_j)_{j\in
J_{k',k''}}=(b_j)_{j\in J_1 \cup J_2 \cup \dots \cup J_{k-1}},$ if $k''=1.$
The second player chooses $l_{k',k''}>i_{k',k''}=i_k$, $m_{k',k''} \in N'_{k'}$, and a
block $w_{m_{k',k''}}\in[y_i:i_k<i\le l_{k',k''}]$ such that
$\|Tx_{m_{k',k''}}-w_{m_{k',k''}}\|<\epsilon_{k',k''}=\epsilon_k$
and $|u_n^k(Tx_{m_{k',k''}})|<\rho_n^k\|Tx_{m_{k',k''}}\|$ for all $n \in F_k.$
The second player also chooses functionals $(v_{k',n}^{k''})_{n\in G_{k',k''}}$ and
positive numbers $(\delta_{k',n}^{k''})_{n\in G_{k',k''}}$ for the first player's
next
turn of game $k'$. For turn $k+1$ of game $\alpha$,
we set  $l_k=l_{k',k''},$ $ m_k=m_{k',k''},$ $
w_{m_k}=w_{m_{k',k''}},$ 
\begin{multline*}
(v_n^k)_{n\in G_k}=
\biguplus_{\substack{0\le \kappa'\le k',\\ \kappa''= k'+k''-\kappa'}}(v_{\kappa',n}^{\kappa''})_{n\in G_{\kappa',\kappa''}}
\biguplus \\
\biguplus_{\substack{k'<\kappa'
\le k'+k'',\\ \kappa''
=k'+k''-\kappa'-1}}(v_{\kappa',n}^{\kappa''})_{n\in G_{\kappa',\kappa''}}
\end{multline*}
and 
\begin{multline*}(\delta_n^k)_{n\in G_k}=
\biguplus_{\substack{0\le \kappa'\le k',\\ \kappa''=
k'+k''-\kappa'}}(\delta_{\kappa',n}^{\kappa''}/2^{k'-\kappa'+1})_{n\in G_{\kappa',\kappa''}}
\biguplus \\
\biguplus_{\substack{k'<\kappa'
\le k'+k'',\\ \kappa''
=k'+k''-\kappa'-1}}(\delta_{\kappa',n}^{\kappa''}/2^{k''+1+\kappa''})_{n\in G_{\kappa',\kappa''}}
\end{multline*}
Observe that $(v_n^k)_{n\in G_k}$ is in fact $(d_{m}^*)_{m\in \{0\}\cup
\{m_i:i \le k\}}.$
This completes turn $k$.

Let $M=\{m_k:k\in \mathbb N\}=\bigcup_{j=0}^\infty M'_j,$ where 
$M'_j=\{m_{j,i}:i\in \mathbb N\}$,  and for each $j \in \mathbb N \cup
\{0\}.$ 
%and let $P_j$ be
%the projection from $Y$ onto $[T x_m: m \in M'_j].$ 
It is easy to see that
because for each $j$, $(x_{m(j,i)})_{i \in \mathbb N}$ was constructed by
using the inductive hypothesis, the order we have used produces a
standard basis of $C_0(\omega^{\alpha})$ with one basis element chosen at
each turn. Also we chose $m_k$ such that
$\|T x_{m_k} -w_{m_k} \| <\epsilon_k$.
We define the projection $P$ onto $[T x_m: m \in M]$ by $T E V$ where $V$ is
the evaluation at $\{d^*_m:m \in M\}\cup \{0\},$ which is homeomorphic
to $[1,\omega^\alpha]$ and can be identified with
$\Gamma=\{\gamma(m):m\in M\}\cup \{\omega^\alpha\}$, and $E$ is the
extension map from $C_0(\Gamma)$ onto $[x_m:m\in M] \subset
C_0(\omega^\alpha).$ The norm of $P$ is at most $2\|T\|\|T^{-1}\|.$
Let $z \in [b_j: j\in J_k,
k\in \mathbb N].$ Fix $k$ and observe that $m_k$ was chosen by the strategy
for game $k'=\gm(k)$. Therefore for $\gm(k)>0,$ 
\begin{multline*}
z\in [b_j:j\in  J_{k',i},i \in
\mathbb N]\\
=[\{b_j:j \in J_{k-k'-i+l}, 1\le l \le k'+i, i =2,3, \dots\}
\\
\cup\{b_j:j\in J_l, 1 \le l \le k\}]
=[b_j:j\in J_i, i\in
\mathbb N],
\end{multline*}
\begin{multline*}
|d_{m_k}^*(z)|=|d_{m(k',\mv(k))}^*(z)|<|d_{m(0,k')}(z)|+(\epsilon_{k',0}/(2\|T\|))\|z\|\\
<|d_0^*(z)|+(\epsilon_{0,0}/(2\|T\|))\|z\|+(\epsilon_0/(2^{2k'+1}\|T\|)\|z\|\\
=|d_0^*(z)|+(\epsilon_0/(2\|T\|))(4^{-1}+4^{-k'})\|z\|.
\end{multline*}
If $k'=0$, we have the simpler estimate
\begin{multline*}
|d_{m_k}^*(z)|=|d_{m(0,\mv(k)}^*(z)|<|d_0^*(z)|+(\epsilon_{0,0}/(2\|T\|))\|z\|\\
=|d_0^*(z)|+(\epsilon_0/(2\|T\|))4^{-1}\|z\|.
\end{multline*}
This shows that the induction hypothesis is satisfied,
Notice that if we pass to a subsequence of $(x_m)_{m\in M}$ that is
equivalent to $(x_n)_{n\in \mathbb N}$ by using Lemma~\ref{equiv-subseq} we
have the same conditions satisfied.
It follows from the estimates above that for $z \in [b_j:j \in J_i, i\in
\mathbb N],$ $\|P z\|\le
\|T\|(\epsilon_0/(2\|T\|))(4^{-1}+4^{-k'})\|z\|<(\epsilon_0/2)\|z\|.$
%We have that $\|P_j z\|\le \epsilon_j \|z\|$
%and that for $m \in M'_j$ $|d^*_{j,m}(z)|\le (\epsilon_j/\|T\|)
%\|z\|.$ For $m \in M'_j$ $d^*_m= d^*_{j,m} + d^*_{m_{0,j}}$ and
%$|d^*_{m_{0,j}}(z)|\le  (\epsilon_0/\|T\|)
%\|z\|.$ Hence $|d^*_m(z)|\le (\epsilon/(2 \|T\|)) \|z\|$ and $\|P z\|\le
%(\epsilon/2) \|z\|.$
Therefore a standard basis of $C_0(\omega^\alpha)$ is $2$-complementably
subsequentially $1$-reproducible.
\end{proof}

We can use the interweaving approach from the argument used in the proof
of the previous proposition to prove the following.

\begin{corollary}\label{interweave}
Suppose that for each $n$, $Z_n$
is a Banach space with a 2-player $D$-complementably subsequentially
$C$-reproducible basis $(z_{n,k})$. Let $Y$ be a
Banach space with a basis and
for each $n$ let $T_n:Z_n\rightarrow Y$ be an
isomorphism such that $\|T_n^{-1}\|\le 1$ and
$\sup \|T_n\| =K <\infty.$ Then for every $\epsilon>0$
for each $n$ there is
a subsequence $(z_{n,k})_{k\in
M_n}$ of $(z_{n,k})$ such that 
$(z_{n,k})_{k\in
M_n}$ is $C$-equivalent to the basis of $Z_n$, $(T_n z_{n,k})_{k\in M_n}$ is a
perturbation of
a block of
the basis of $Y$, disjointly supported from  the blocks for
$(T_m z_{m,k})$, all $m\ne
n$, and
there is a projection $P_n$, $\|P_n\| \le K D,$ from  $[T_m z_{m,k}: m\in \mathbb N, k \in M_m]$ onto $[T_n z_{n,k}:k\in M_n]$ such that
for any $z \in [T_m z_{m,k}: m\in \mathbb N, k \in M_m, m\ne n]$, $\|Pz\|<\epsilon
\|z\|.$ Moreover the sequences $(z_{n,k})_{k\in
M_n}$ are produced by a 2-player
game.
\end{corollary}

\begin{proof} (Sketch) Let $\epsilon>0,$ and let
$\epsilon_k=\epsilon/(K2^{2k+2}),$ for all $k \in \mathbb N.$ Without loss
of generality we may assume that the basis of $Y$ is bi-monotone. Let game 0 be
the overall game with players 1 and 2. For each $n\in \mathbb N$ we will
have a game $n$ whose second player is following the strategy to produce
the required subsequence of $(z_{n,k})_{k \in \mathbb N}.$ As in the
previous proof we will use the first subscript on parameters
to denote the game.

 We begin with player 1 as the first player in game 0. Player 1 chooses
$i_{0,1},$ a finite sequence $(u_{0,n}^1)_{n\in F_{0,1}}\subset Y^*,$
positive
numbers $(\rho_{0,n}^1)_{n\in F_{0,1}},$ and a finite set of blocks
$(b_{0,j})_{j\in J_{0,1}}\subset [y_i: i\le i_{0,1}].$ Player 2 views
this
as the move of the first player in turn 1 of game 1 with
$\epsilon_{i,k}=\epsilon_1/2^{2(k+1)}, $ for $k \in \{0\}\cup\mathbb N,$ 
$i_{1,1}=i_{0,1},$ $(u_{1,n}^1)_{n\in F_{1,1}}=(u_{0,n}^1)_{n\in
F_{0,1}},$
$(\rho_{1,n}^1)_{n\in F_{1,1}}=(\rho_{0,n}^1)_{n\in F_{0,1}},$
and $(b_{1,j})_{j\in J_{1,1}}=(b_{0,j})_{j\in J_{0,1}}.$
Using the strategy to produce a subsequence of $(z_{1,k})_{k\in \mathbb N}$ 
player 2 chooses an integer $l_{1,1},$ a finite
set $M_{1,1}\subset \mathbb N,$ a finite set of blocks $(w_{1,m}^1)_{m\in
M_{1,1}}\subset [y_i:i_{1,1}<i\le l_{1,1}]$ with $\sum_{m\in M_{1,1}} \|T_1
z_{1,m}-w_{1,m}\|<\epsilon_1, $ and with $|u_{1,n}^1(x)|<\rho_{1,n}^1\|x\|$
for all $x \in [T_{1} x_m:m\in M_{1,1}],$ a finite sequence $(v_{1,n}^1)_{n\in
G_{1,1}}\subset Y^*,$ and positive numbers $(\delta_{1,n}^1)_{n\in
G_{1,1}}.$  Let $l_{0,1}=l_{1,1},$ $(v_{0,n}^1)_{n\in
G_{0,1}}=(v_{1,n}^1)_{n\in
G_{1,1}},$
and $(\delta_{0,n}^1)_{n\in
G_{0,1}}=(\delta_{1,n}^1)_{n\in
G_{1,1}}.$

For the second turn of game 0, Player 1 chooses
$i_{0,2},$ a finite sequence $(u_{0,n}^2)_{n\in F_{0,2}}\subset Y^*,$
positive
numbers $(\rho_{0,n}^2)_{n\in F_{0,2}},$ and a finite set of blocks
$(b_{0,j})_{j\in J_{0,2}}\subset [y_i: l_{0,1}<i\le i_{0,2}],$ such that
$|v_n^1(b)|<\delta_n^1\|b\|$ for all $b\in [b_{0,j}:j\in J_{0,2}],$ $n\in
G_{1,1}.$ Player 2 views
this
as the move of the first player in turn 2 of game 1, i.e.,
$i_{1,2}=i_{0,2},$ $(u_{1,n}^2)_{n\in F_{1,2}}=(u_{0,n}^2)_{n\in
F_{0,2}},$
$(\rho_{1,n}^2)_{n\in F_{1,2}}=(\rho_{0,n}^2/2)_{n\in F_{0,2}},$
and $(b_{1,j})_{j\in J_{1,2}}=(b_{0,j})_{j\in J_{0,2}}.$
Player 2 chooses an integer $l_{1,2},$ a finite
set $M_{1,2}\subset \mathbb N,$ a finite set of blocks $(w_{1,m}^2)_{m\in
M_{1,2}}\subset [y_i:i_{1,2}<i\le l_{1,2}]$ with $\sum_{m\in M_{1,2}} \|T_1
z_{1,m}-w_{1,m}\|<\epsilon_{1,2}, $ and with $|u_{1,n}^2(x)|<\rho_{1,n}^2\|x\|$
for all $x \in [T_{1} x_m:m\in M_{1,2}],$ a finite sequence
$(v_{1,n}^2)_{n\in
G_{1,2}}\subset Y^*,$ and positive numbers $(\delta_{1,n}^2)_{n\in
G_{1,2}}.$ Let $l_{0,2}=l_{1,2},$ $(v_{0,n}^2)_{n\in
G_{0,2}}=(v_{1,n}^2)_{n\in G_{1,2}},$
and $(\delta_{0,n}^2)_{n\in G_{0,2}}=(\delta_{1,n}^2/2)_{n\in G_{1,2}}.$

For the third turn of game 0 Player 1 chooses
$i_{0,3},$ a finite sequence $(u_{0,n}^3)_{n\in F_{0,3}}\subset Y^*,$
positive
numbers $(\rho_{0,n}^3)_{n\in F_{0,3}},$ and a finite set of blocks
$(b_{0,j})_{j\in J_{0,3}}\subset [y_i: l_{0,2}<i\le i_{0,3}],$ such that
$|v_n^2(b)|<\delta_n^2\|x\|$ for all $b \in [b_{0,j}:j\in J_{0,3}],$ $n \in
G_{1,2}.$

Player 2 now begins game 2 by setting
$\epsilon_{2,k}=\epsilon_2/2^{2(k+1)}, $ $k \in \{0\}\cup \mathbb N, $ 
$i_{2,1}=i_{1,3},$ $(b_j)_{j\in
J_{2,1}}=(b_j)_{j\in J_{1,1}} \uplus (w_{1,m})_{m\in M_{1,1}} \uplus 
(b_j)_{j\in J_{1,2}} \uplus (w_{1,m})_{m\in M_{1,2}} \uplus (b_{0,j})_{j\in
J_{0,3}},$ $(u_{2,n}^1)_{n\in
F_{2,1}}=(u_{0,n}^3)_{n\in
F_{0,3}} \uplus (v_{1,n}^2)_{n\in
G_{1,2}},$ and $(\rho_{2,n}^1)_{n\in
F_{2,1}}=(\rho_{0,n}^3)_{n\in F_{0,3}} \uplus (\delta_{1,n}^2/2)_{n\in
G_{1,2}}.$
Using the strategy for $(z_{2,n})_{n\in \mathbb N}$, Player 2 chooses an integer
$l_{2,1},$ a finite
set $M_{2,1}\subset \mathbb N,$ a finite set of blocks $(w_{2,m}^k)_{m\in
M_{2,1}}\subset [y_i:i\le l_{2,1}]$ with $\sum_{m\in M_{2,1}} \|T_2
z_{2,m}-w_{2,m}\|<\epsilon_{2,1}, $ and with $|u_{2,n}^1(x)|<\rho_{2,n}^1\|x\|$
for all $x \in [T_{2} z_{2,m}:m\in M_{2,1}],$ a finite sequence $(v_{2,n}^1)_{n\in
G_{2,1}}\subset Y^*,$ and positive numbers $(\delta_{2,n}^1)_{n\in
G_{2,1}}.$

Continuing in this way player 2 works through the next turns of the games in
progress and then starts the next new game. Player 2 adjusts the parameters
so that in each round of turns the estimates sum properly.
\end{proof}

By using a modification of the projection from the previous result we can
position the sequences $(T_m z_{m,k})_{k\in K_m},$ $m\in \mathbb N$, $m \ne n,$ in 
the kernel of a projection onto $[T_n z_{n,k}:k\in K_n].$ We believe that
the next lemma is well-known and is due to Pelczynski but we do not have a precise
reference.

\begin{lemma}\label{perturbP}  Suppose that $X$ is a Banach space,
$X_1$ and
$Z$ are subspaces of $X,$ $0<\epsilon<1,$ and $P$ is a projection from $x$
onto $X_1$ and $\|Pz\| \le \epsilon \|z\|,$ for all $z \in Z.$ Then there
is a projection $Q$ from $X_1+Z$ onto $X_1$ such that $Qz=0$ for all $z \in
Z.$
\end{lemma}
\begin{proof}
Let $R=(I-P)|_Z$ and let $W=(I-P)Z.$  Because $\epsilon <1$, $\|Rz\| \ge
(1-\epsilon)\|z\|$ for all $z\in Z.$ Thus  $R$ is an isomorphism from $Z$
onto $W$. Let $Q=I-R^{-1}(I-P).$ $Qx=x$ for all $x \in X_1$ and $Qz=0$ for
all $z \in Z.$
\end{proof}

\section{An ordinal index for unconditional sums}
An essential technical tool that we need is an ordinal index that is closely related
the basis index introduced by Bourgain.  The new index, however, will
be defined for an unconditional sum of Banach spaces rather than for
one dimensional subspaces (i.e., a basis).
%Recall that  a basis $(z_n)$ is $K$-unconditional if the basis
%satisfies $0\le a_n \le b_n$ all $n$ implies
%$\|\sum_n a_n z_n \| \le K\|\sum_n
%b_n z_n \|$.

Recall that given a set $A$, a tree $\mathcal T$ on $A$
is a partially ordered family of finite tuples of elements
of $A$ such that if $(a_1,a_2,\dots,a_n) \in \mathcal T$ then
$(a_1,a_2,\dots,a_{n-1})\in \mathcal T$. The partial order is by extension
so that $(a_1,a_2,\dots,a_n) \le (b_1,b_2,\dots,b_k)$ iff $n\le k$ and
$a_j=b_j$ for $j=1,2,\dots, n$. A branch of $\mathcal T$ is a maximal
totally ordered subset. The tuples in the tree are called nodes, the
element $a_j$ of the node $(a_1,a_2,\dots,a_n)$ will be the $j$th entry,
and if there is a minimal node $r$ that is comparable to all nodes,
we will say that the tree is rooted at $r$.

A derivation of $\mathcal T$ is defined by deleting nodes with no
proper extensions. Let $\mathcal T^{(1)}=\{b\in \mathcal T :\exists c
\in \mathcal T, b\le c, b\ne c\}$. If $\mathcal T^{(\alpha)}$ has been
defined, let $\mathcal T^{(\alpha+1)}=(\mathcal T^{(\alpha)})^{(1)}$. If
$\beta$ is a limit ordinal, $\mathcal T^{(\beta)}=\bigcap_{\alpha<\beta}
\mathcal T^{(\alpha)}.$ If $\alpha$ is the smallest countable ordinal such that
$\mathcal T^{(\alpha)}=\emptyset$, then the order of the tree is set as
$o(\mathcal T)=\alpha$. Otherwise we put $o(\mathcal T)=\omega_1$.
If $A$ is a separable complete metric space and $\mathcal T$ on $A$ is
closed and
has order $\omega_1$, then $T$ has an infinite branch. 
%(reference? Does this require some closedness condition?)

Suppose that $A$ is a separable Banach space and $(y_i)$ is a normalized basis of a Banach space
$Y$. If $K <\infty$ and the nodes are
finite
normalized sequences $(a_i)_{i=1}^n$ in $A$ that are $K$-equivalent to
the initial segment $(y_i)_{i=1}^n,$
the index gives a
very useful tool for checking whether the Banach space $Y$
embeds into the Banach space $A$. Indeed,
for this Bourgain basis tree
one
only needs to check that the index is $\omega_1$. Unfortunately this tree
does not seem to be suitable for handling sums of spaces, so we need to
make some major modifications.

Let $Z$ be a Banach space with a $1$-unconditional basis $(z_n)$ and
for each $n$ let
$Y_n$ be a Banach space with norm $\|\cdot\|_n$. Then by $(\sum Y_n)_Z$
we denote the direct sum of $Y_n$'s with respect to $(z_n)$. That is,
the space is  the Banach space of sequences $(y_n)_{n=1}^\infty,$ $y_n
\in Y_n$ for all $n$, with finite norm $\|(y_n)\|_Z=\|\sum_{n=1}^\infty
\|y_n\|_n z_n\|_Z.$ Observe that this $Z$-sum  is well-behaved with respect
to uniformly bounded sequences of operators acting on the coordinate
spaces.
%Note that if we replace each $Y_n$ by an isomorphic space $Y'_n$ and
%the constants are uniformly bounded in $n$, then an isomorphic copy
%$(\sum Y'_n)_Z$ is produced.

Let $Z$ and $Y_n$, $n \in \mathbb N,$ be as above and, in addition,
fix constants $C,D>0$ and a Banach space $X$.
Consider a tree $\mathcal T$
of tuples consisting of pairs of subspaces and isomorphisms
$$((X_1,T_1),(X_2,T_2),\dots,(X_k,T_k)),$$
where $X_j$ is a
subspace of $X$ and $T_j$ is an isomorphism from $X_j$ onto $Y_j$
such that
$\|T_j\|\le C,$ $\|T_j^{-1}\|\le 1,$ and for all $x_j \in X_j$, we have
$$\left\|\sum_{j=1}^k x_j\right\|\le \|(T_j x_j)\|_Z\le D
\left\|\sum_{j=1}^k x_j\right\|, \ 1\le j\le k.$$
We partially order $\mathcal T$ by extension and the order of the tree
is defined as before.
We call $\mathcal T$ a $(\sum Y_n)_Z$-tree in $X$ with constants $C,
D$, and the order of the tree is referred as $(\sum Y_n)_Z$ index.

Even if we assume that all of the spaces are separable, we cannot proceed as
before to
establish that trees with index $\omega_1$ have an infinite branch. We do
not know whether this is even true. However we are able to prove the
following result which is satisfactory for our purposes. The proof
actually
shows that there is an infinite branch which in a sense close to branches
of the given tree. 

\begin{theorem}\label{index_omega1} Let $Z$ be a Banach space with a normalized
1-uncondi\-tional basis $(z_n)$, and let $X$ and $Y_n, n\in \mathbb N$ 
be separable
Banach spaces. If $\mathcal T$ is a $(\sum Y_n)_Z$-tree in $X$ with index
$\omega_1$ and constants $C,D$, then  $X$ contains a subspace which is
$D$-isomorphic to $(\sum Y_n)_Z$.
\end{theorem}

\begin{proof} Let $W$ be a Banach space with a basis $(w_k)_{k=1}^\infty$
that contains $X.$
%and let $P_n$ denote the basis projection of $W$ onto $[w_k: k\le n].$
For each $j,k \in \mathbb N$ choose a subset
$\Omega_{j,k}$ of $[w_i: i\le j]\cap B_W,$ where $B_W$ denotes the unit
ball of $W$, such that
for all $w \in [w_i: i\le j]\cap B_W,$
there exists $w' \in \Omega_{j,k}$ such that
$\|w-w'\|<2^{-k}.$
For each $n,m\in \mathbb N$ let $\Upsilon_{n,m}$ be a finite
subset of $B_{Y_n}$ such that $\Upsilon_{n,m}\subsetneq \Upsilon_{n,m+1}$
for all $m$ and $\bigcup_{m=1}^\infty \Upsilon_{n,m}$ is dense in
$B_{Y_n}.$

Let the branches of $\mathcal T$ be indexed by $A$ so that for each
$\alpha
\in A,$ $((X_{\alpha,1},T_{\alpha,1}), \dots,
(X_{\alpha,k},T_{\alpha,k}))$, $1\le k< M_\alpha,$ $M_\alpha \le
\infty,$ is a branch and if $M_\alpha<\infty$, then
$((X_{\alpha,1},T_{\alpha,1}), \dots,
(X_{\alpha,M_\alpha-1},T_{\alpha,M_\alpha-1}))$ is a terminal node.
For each $(X_{\alpha,1},T_{\alpha,1})$, the initial node of the branch
$\alpha \in A$, find $n=n(\alpha,1,1)$ such that for each $y\in
\Upsilon_{1,1}$ there exists $w(y) \in \Omega_{n,2}$ such that
$\|w(y)-T_{\alpha,1}^{-1}y\|<2^{-1}$ and $w$ is one-to-one.
We may assume that if
$(X_{\alpha,1},T_{\alpha,1})= (X_{\gamma,1},T_{\gamma,1})$ then
$n(\alpha,1,1)=n(\gamma,1,1).$ For each $\alpha$  let
$W(\alpha,1,1)=\{w(y): y\in \Upsilon_{1,1}\}.$ For each $y \in
\Upsilon_{1,1}$ there may be more than one
possibility for $w(y)$
but we choose one to put in the set and use
the same
choice for all $\alpha$ with the same initial node. ($w(\alpha,1,1)(y)$ actually
depends on $\alpha$ and the index $(1,1)$ but we are suppressing this
in the notation for now.)
There are countably many possibilities for integers $n(\alpha,1,1),$
sets $W(\alpha,1,1)$, and bijections $w(\alpha,1,1):\Upsilon_{1,1}\rightarrow
W(\alpha,1,1),$
so there must be an integer $n(1,1)$, a bijection
$w(1,1)$ and a subset
$W(1,1)$ of $\Omega_{n(1,1),2}$ such that if
\begin{multline*}
A_1=\{\alpha:
n(\alpha,1,1)=n(1,1) \text{ and }W(\alpha,1,1)=W(1,1),\\
w(\alpha,1,1)(y)=w(1,1)(y) \text{ for all }y\in \Upsilon_{1,1}\},
\end{multline*}
the subtree $\mathcal T_1$ of
nodes with first entry $(X_{\alpha,1},T_{\alpha,1)})$ for $\alpha \in A_1$
has index $\omega_1.$
This completes the first step of an induction.

Next for each $((X_{\alpha,1},T_{\alpha,1}),(X_{\alpha,2},T_{\alpha,2})),$
which occurs as the second node of   a branch $\alpha \in A_1,$
we find integers $n(\alpha,1,2)$ and $n(\alpha,2,1)$ and subsets
$W(\alpha,2,1)$ of $\Omega_{n(\alpha,2,1),2}$ and $W(\alpha,1,2)$ of
$\Omega_{n(\alpha,1,2),3}$ such that for each $y_2\in
\Upsilon_{2,1}$, $y_1 \in \Upsilon_{1,2}$, there are  $w(y_2) \in
W(\alpha,2,1)$, $w(y_1) \in W(\alpha,1,2)$ such that
$\|T_{\alpha,2}^{-1}(y_2)-w(y_2)\|<2^{-1}$ and
$\|T_{\alpha,1}^{-1}(y_1)-w(y_1)\|<2^{-2}$. Here we again assume that
$w(\cdot)$
is a bijection on each set and that the integers $n(\alpha,1,2)$ and
$n(\alpha,2,1)$ and $w(\cdot)$ depend only on the second node in the
branch,
$((X_{\alpha,1},T_{\alpha,1}),(X_{\alpha,2},T_{\alpha,2}))$. As before
there are countably many choices for
the integers, bijections and finite sets, so there are integers $n(2,1)$,
$n(1,2)$, bijections $w(2,1)$ and $w(1,2)$, and
subsets $W(2,1)$ of $\Omega_{n(2,1),2}$ and $W(1,2)$ of
$\Omega_{n(1,2),3}$ such that if
\begin{multline*}
A_2=\Big\{\alpha \in
A_1:n(\alpha,1,2)=n(1,2), n(\alpha,2,1)=n(2,1), \\
W(\alpha,1,2)=W(1,2),
w(\alpha,1,2)(y)=w(1,2)(y)\text{ for all }y\in \Upsilon_{1,2},\\
W(\alpha,2,1)=W(2,1),w(\alpha,2,1)(y)=w(2,1)(y)\text{ for all }y\in
\Upsilon_{2,1}\Big\},
\end{multline*}
then the tree $\mathcal T_2$ of nodes
from the branches in $A_2$ has index $\omega_1.$

Continuing in this way we find a decreasing sequence of subsets
$(A_m)_{m=1}^\infty$ of $A,$ positive integers $n(i,j),$ and subsets
$W(i,j)$ of $\Omega_{n(i,j),j+1},$ $i,j \in \mathbb N,$
such that if $i+j\le m+1$, there is a bijection $w(i,j)$ from
$\Upsilon_{i,j}$ onto $W(i,j)$ with
$\|T_{\alpha,i}^{-1}(y)-w(i,j)(y)\|<2^{-j}$ for all $y \in
\Upsilon_{i,j}$.
Moreover, the subtree $\mathcal T_m$ of nodes from branches in
$A_m$ has order $\omega_1.$

Because $\Upsilon_{i,j} \subset \Upsilon_{i,k}$, all $k>j,$ for each
$i \in \mathbb N$, we can define
$$B_i=\Big\{x: \text{there exists }y \in
\Upsilon_{i,j},j \in \mathbb N,  \lim_k  w(i,k)(y)=x\Big\}.$$
(The limit is in the norm of $W$.)
Notice that for $x \in B_i$,
for each $\alpha\in A_m, j\le k \le m+1-i,$
if $x_{\alpha,i}=T_{\alpha,i}^{-1}(y),$
$$\|x_{\alpha,i}-w(i,k)(y)\|=\|T_{\alpha,i}^{-1}(y)-w(i,k)(y)\|<2^{-k}.$$
Thus for any choice $\alpha(m)\in A_m$, $m>i,$ $\lim_{m \rightarrow \infty} x_{\alpha(m),i} =x\in X$, 
and $\|x\| \le 1.$ Further
$\|T_{\alpha,i}^{-1}(y)\|\ge C^{-1} \|y\|$ so that $\|x\|\ge C^{-1}
\|y\|.$

If $y\in \Upsilon_{i,j}$ for some $j$, then for all $\alpha \in A_m$,
$\|T_{\alpha,i}^{-1}(y)-w(i,k)(y)\|<2^{-k},$ for $j\le k \le m+1-i.$ Thus
$\|w(i,k)(y)-w(i,k')(y)\|\le 2^{-\min(k,k')+1}$ if $j\le k' \le m+1-i$
also. Thus we may define a map $T^{-1}_i: \bigcup_j
\Upsilon_{i,j}\rightarrow B_i$
by $T^{-1}_i(y)=\lim_k w(i,k)(y).$ On its domain
$T^{-1}_i$ is the pointwise limit of
uniformly bounded linear maps, thus $T^{-1}_i$ continuously extends to
$B_{Y_i}$ as an affine map and to all of $Y_i$ by scaling. Because it is
bounded below, it is an isomorphism. Let $X_i$ be the range of $T^{-1}_i$.

It remains to verify that $[x:x\in X_i, i\in \mathbb N]$ is isomorphic to
$(\sum_i Y_i)_Z.$ This follows from the fact that for any finite sequence
$(x_i)_{i=1}^k$, $x_i \in B_{X_i}$ and $\epsilon>0$,  we can find $j$
such that for
each $i\le k$ there exists $y_i \in \Upsilon_{i,j}$ such that
$$\|T^{-1}_i(y_i)-x_i\|<\epsilon k^{-1}.$$
For $m$ sufficiently large if
$\alpha\in A_m$ then
$$\|T^{-1}_i(y_i)-T^{-1}_{\alpha,i}(y_i)\|<\epsilon k^{-1}$$ for $1\le i
\le k.$
Thus
$$\Big\|\sum_{i=1}^k x_i - \sum_{i=1}^k T^{-1}_{\alpha,i}(y_i)\Big\| <
2 \epsilon $$
and, because we have assumed $(z_i)$ is normalized,
\begin{multline} \label{uncdiff} \Big\|\sum_{i=1}^k \|T_ix_i\|z_i -
\sum_{i=1}^k \|y_i\|z_i\Big\|\le \sum_{i=1}^k \|T_ix_i-y_i\| \\
\le C\sum_{i=1}^k \|x_i -T^{-1}_i
y_i\|<2 C \epsilon.\end{multline}
We have that for $\alpha \in A_m,$  $m>j+k+1,$
$$D^{-1} \Big\|\sum_{i=1}^k \|y_i\|z_i\Big\| \le  \Big\|\sum_{i=1}^k
T^{-1}_\alpha y_i\Big\|
\le \Big\|\sum_{i=1}^k \|y_i\|z_i\Big\|$$ and consequently
\begin{multline*}
D^{-1}\Big\|\sum_{i=1}^k \|y_i\|z_i\Big\|-2\epsilon  \le
\Big\|\sum_{i=1}^k T^{-1}_\alpha y_i\Big\| -2 \epsilon  \le
\Big\|\sum_{i=1}^k x_i \Big\| \le \\
\Big\|\sum_{i=1}^k T^{-1}_\alpha y_i\Big\|+2\epsilon
 \le  \Big\|\sum_{i=1}^k \|y_i\|z_i\Big\|+2\epsilon.
\end{multline*}
Using the estimate (\ref{uncdiff})  We obtain
\begin{multline*}
D^{-1}\Big\|\sum_{i=1}^k \|T_i x_i\|z_i\Big\| -2 C \epsilon D^{-1}
-2 \epsilon \\
 \le \Big\|\sum_{i=1}^k x_i \Big\| \le  \Big\|\sum_{i=1}^k \|T_i
x_i\|z_i\Big\|+2 C \epsilon  +2 \epsilon.
\end{multline*}
Since $\epsilon>0$ is arbitrary, we have
$$
D^{-1}\Big\|\sum_{i=1}^k \|T_i x_i\|z_i\Big\|\le \Big\|\sum_{i=1}^k
x_i \Big\|
\le \Big\|\sum_{i=1}^k \|T_ix_i\|z_i\Big\|,$$ proving the result.
\end{proof}

\section{$C(K)$ subspaces of Elastic Spaces}

In the previous sections we have developed the tools that will allow us to
generalize the argument of Johnson and Odell and prove the following.

\begin{proposition} \label{c0sum} Let $K,C,D \ge 1$ be constants.
Suppose that $X$ is a separable $K$-elastic space
and $(Y_n)_{n=1}^\infty$ is a sequence of spaces with
two-player $D$ complementably $C$-reproducible bases that embed into
$X$. Then $X$ contains a
subspace isomorphic to $\big(\sum_{n=1}^\infty Y_n\big)_{c_0}.$
\end{proposition}

\begin{proof}
We can assume that the elastic space $X$ is contained in a Banach space
$W$ with a bi-monotone basis $(w_n)$. We will show by an inductive
construction that for every $\epsilon>0$ and each countable limit ordinal
$\alpha$, $X$ contains a $\big(\sum_{n=1}^\infty Y_n\big)_{c_0}$-tree
of order at least
$\alpha$ with both constants $K(1+\epsilon)$. To do this we will construct
Banach spaces $V^\alpha$, $\alpha<\omega_1$ which are isomorphic to
subspaces of $X$ and contain
$\big(\sum_{n=N}^\infty Y_n\big)_{c_0}$-trees of order $\alpha$ for each
$N\ge 1$.

In order to avoid cluttering the arguments with multiplication by
various values of the form $(1+\delta)$ that are incurred from tiny
perturbations of block bases, we will suppress these throughout.

For each $n \in \mathbb N$, let $(y_{n,k})_{k=1}^\infty$ be a two-player
$D$
complementably $C$-reproducible basis for $Y_n$, and let $T_n$ be an
isomorphism from $Y_n$ into $X$ with $\|T_n\|\le K$ and $\|T_n^{-1}\|\le
1.$
Let $0<\delta<1.$
By interweaving the two-player games (See Corollary \ref{interweave}.)
we can find subsequences $(y_{n,k})_{k\in K_n}$ of
$(y_{n,k})_{k=1}^\infty$
for all $n\in \mathbb N,$ such that $(y_{n,k})_{k\in K_n}$
is equivalent to $(y_{n,k})_{k\in \mathbb N},$ for all $n\in \mathbb N,$
$(T_n y_{n,k})_{k\in K_n,n\in
\mathbb N}$
is (equivalent to) a block of $(w_j),$ and for each $m$ there is a
projection $P_m$ from $V=[T_n y_{n,k}:k\in K_n,n\in \mathbb N]$ onto
$[T_m y_{m,k}:k\in K_m]$ with $\|P_m y\|\le \delta \|y\|$ for all $y \in
[T_n y_{n,k}:k\in K_n,n\in \mathbb N, n\ne m].$ Let $R_n$ be the basis
equivalence from $(y_{n,k})_{k\in K_n}$ to $(y_{n,k})_{k=1}^\infty$.
By Lemma \ref{perturbP}, for each $m$ we may
assume that the projection $P_m$ is zero on $[T_n y_{n,k}:k\in K_n,n\in
\mathbb N, n\ne m]$.

For each $i \in \mathbb N$ define a norm $\|\cdot\|_i$ on $V$ by
$$\|y\|_i= \sup\Big\{\|R_m T_m^{-1}P_m y\|:m\in \mathbb N\Big\}\vee
\frac{\|y\|}{i K C }.$$
Let $B = \sup_i \|R_i T_i^{-1} P_i\|.$  Then $B\le C D'$ where
$D'=(1+\delta)^{-1}(2+D)$. (See the proof of Lemma~\ref{perturbP}.) Clearly
$\frac{\|y\|}{i K C} \le \|y\|_i \le B \|y\|$
for all $y \in V.$ Notice that if $y \in
[T_j y_{j,k}:k\in K_j]$ for some $j$ then $\|y\|_i=\|R_j T_j^{-1}y\|.$
Thus
$[T_j y_{j,k}:k\in K_j]$ with norm $\|\cdot\|_i$ is isometric to $Y_j.$
For each $i \in \mathbb N$ let $V^i$ be the space $V$ with norm $\|\cdot
\|_i.$

It is clear from the construction that if $F\subset\mathbb N$ with
$|F|\le i$, $v=\sum_{j\in F} v_j$ and $v_j \in [T_j y_{j,k}:k\in K_j]$,
then $\|v\|_i=\max \{\|v_j\|_i:j \in  F\}$. Indeed,
\begin{eqnarray*}
\|v\|_i&=&\sup\Big\{\|R_m T_m^{-1} P_m v\|:m\in \mathbb N\Big\}\vee
\frac{\|v\|_X}{i K C} \\ &\le& \max \Big\{\|R_m T_j^{-1} v_j\|:j \in  F\Big\}
\vee \frac{1}{i K C}\sum_{j\in F} \|v_j\|_X\\ &=& \max \{\|v_j\|_i:j \in  F\}.
\end{eqnarray*}
Conversely,
\begin{eqnarray*}
\|v\|_i&\ge& \sup\Big\{\|R_m T_m^{-1} P_m v\|:m\in \mathbb N\Big\}\ge \max
\Big\{\|R_m T_j^{-1} v_j\|:j \in  F\Big\}\\
&\ge& \max \{\|v_j\|_i:j \in  F\}.
\end{eqnarray*} 
This shows if $F=\{N,N+1,\dots,N+p\}$, $p<i,$ $(([T_j y_{j,k}:k\in
K_j],R_jT^{-1}_j))_{j\in F}$ is a node of a $\big(\sum_{n=N}^\infty
Y_n\big)_{c_0}$-tree.

Note that the definition of the norm on $V^i$ produces a hereditary
property. If for each $n \in \mathbb N,$ we pass to a
subsequence of $(y_{n,k})_{k\in K_n}$ which is equivalent, then the nodes
in $X$ formed from pairs with first coordinate the image under $T_n$ of
the closed span of the
subsequence, we again will have a node of a $\big(\sum_{n=N}^\infty
Y_n\big)_{c_0}$-tree.

Thus for all $N\ge 1$
there are branches of a $\big(\sum_{n=N}^\infty Y_n\big)_{c_0}$-tree
of length
$i$ in $V^i$. Since $X$ is $K$-elastic, for each $i \in \mathbb N$
there is an
isomorphism $S_i$ from $V^i$ into $X$ such that $\|v\|_i \le \|S_i v\|
\le K \|v\|_i$ for all $v \in V^i.$
Thus $X$ has the tree index (with constants $K$) at least
$\omega$.

We now extend this a little by observing that we can embed substantial
parts of the spaces $V^i$ simultaneously into $X$ as disjointly
supported blocks of $(w_j)$. Let $\delta >0$ and let
$(y_{i,n,k})_{k=1}^\infty$ be a two-player
complementably sequentially reproducible basis of
the subspace of $V^i$ spanned by $(Ty_{n,k})_{k\in K_n}$ (in norm
$\|\cdot\|_i$). By the definition
of $\|\cdot \|_i$ this subspace is isometric to $Y_n.$ By interweaving the
two player games for $(y_{i,n,k})_{k=1}^\infty$, $i,n \in \mathbb N,$
we can find subsequences $(y_{i,n,k})_{k\in K_{i,n}}$ of
$(y_{i,n,k})_{k=1}^\infty$
for all $i,n\in \mathbb N,$ such that $\{S_i y_{i,n,k}: k\in
K_{i,n},i,n\in \mathbb
N\}$
is (equivalent to) a block of $(w_j)$ in some order, and for each $j,m$
there is a
projection $P_{j,m}$ from $V^{\omega}:=[S_i y_{i,n,k}:k\in K_{i,n},i,n\in
\mathbb N]$ onto
$[S_j y_{j,m,k}:k\in K_{j,m}]$ with $\|P_{j,m} y\|\le \delta \|y\|$
for all $y \in
[S_i y_{i,n,k}:k\in K_{i,n},i,n\in \mathbb N, (i,n)\ne (j,m)].$ Let
$R_{i,n}$ be the basis
equivalence from $(y_{i,n,k})_{k\in K_{i,n}}$ to
$(y_{i,n,k})_{k=1}^\infty$.
By Lemma \ref{perturbP}, for
all $j,m$ we may
assume that the projection $P_{j,m}$ is zero on $[S_i y_{i,n,k}:k\in
K_n,i,n\in
\mathbb N, (i,n)\ne (j,m)]$.

Define a norm on $V^\omega$
by
$$\|y\|_\omega= \sup\Big\{\|R_{j,m} S_{j}^{-1}P_{j,m} y\|:j,m\in \mathbb
N\Big\}\vee
\frac{\|y\|}{K C },$$
where $B = \sup_{j,m} \|R_{j,m} S_{j}^{-1} P_{j,m}\|.$
As before $\frac{\|y\|}{K C } \le \|y\|_\omega \le B \|y\|$ for all $y \in
V^{\omega}.$ 
%Let $V_{\omega}$ be $V^{\omega}$ with the norm $\|\cdot \|_{\omega}.$
For each $i$, if $|F| \le i$ and for each $n\in F, $
$x_{i,n} \in S_{i} [y_{i,n,k}:k\in K_{i,n}],$
then
$\|\sum_{n \in F} x_{i,n}\|_\omega =\sup_{n\in F} \|R_{i,n}
S_{i}^{-1}x_{i,n}\|=\|\sum_{n \in F} R_{i,n}S_{i}^{-1}x_{i,n}\|_i.$
Consequently, the nodes of  length less than or equal to $i$ of the
$\big(\sum_{n=N}^\infty Y_n\big)_{c_0}$-tree of $V^i$ are nodes of the
$\big(\sum_{n=N}^\infty Y_n\big)_{c_0}$-tree in the corresponding subspace
of $V^\omega$. In particular, $V^\omega$ contains a
$\big(\sum_{n=N}^\infty
Y_n\big)_{c_0}$-tree of order $\omega$ for each $N\ge 1$.

Using the property that $X$ is $K$ elastic, we can find a subspace of $X$
$K$-isomorphic to $V^\omega.$ Observe that $V^\omega$ has the same
hereditary property with respect to equivalent subsequences of the
$(y_{i,n,k})$ that $V^i$ has. This completes the first part of the
induction.

Now assume that for all limit ordinals $\beta<\alpha$ we have constructed
spaces $V^\beta$
containing a $\big(\sum_{n=N}^\infty Y_n\big)_{c_0}$-tree of order $\beta$
for each $N \ge 1$ and that $V^\beta$ is isomorphic to a subspace of
$X$. We assume that this tree property is preserved under passing to equivalent
subsequences of the bases of each coordinate in each node.  Further we 
can assume that the tree has only countably many nodes and
hence that there are countably many subspaces $Y_{n,m}$, $n,m \in \mathbb
N$ such that $Y_{n,m}$ is isometric to $Y_n$ for each $m \in \mathbb N$
and that each
node of the tree is of the form
$$\big((Y_{N,m(N)},J_{N,m(N)}), \dots,
(Y_{N+k,m(N+k)},J_{N+k,m(N+k)})\big).$$
If we need to refer to these subspaces
for more than one $\beta$, then we will add a superscript $\beta$, e.g.,
$Y_{2,3}^\beta.$ A two-player complementably sequentially reproducible basis
for $Y_{n,m}$ will be denoted $(y_{n,m,k})_{k=1}^\infty.$

There are two cases to consider. If $\alpha$ is a limit ordinal of
the form
$\beta+\omega$, let $T^\omega$ be an isomorphism of $V^\omega$  into $X$
and let
$T^\beta$ be an isomorphism of $V^\beta$ into $X$ as given by the elastic
property. By interweaving the two-player games for the bases
$(y^\gamma_{n,m,k})_{k=1}^\infty$ for each $n,m \in \mathbb N,$  and
$\gamma=\omega, \beta,$
we can find subsequences $(y^\omega_{n,m,k})_{k\in K^\omega_{n,m}}$ of
$(y^\omega_{n,m,k})_{k=1}^\infty$,
$(y^\beta_{n,m,k})_{k\in K^\beta_{n,m}}$ of
$(y^\beta_{n,m,k})_{k=1}^\infty$,
for all $n,m\in \mathbb N,$ such that
$\big\{T^\gamma y^\gamma_{n,m,k}: k\in K^\gamma_{n,m},n,m\in
\mathbb N, \gamma=\omega,\beta\big\}$
is (equivalent to) a block basis of $(w_j)$ in some order, and for each
$\eta\in \{\omega,\beta\},i,j\in \mathbb N$ there is a
projection $P^\eta_{i,j}$ from
$V^{\beta+\omega}:=[T^\gamma y^\gamma_{n,m,k}:k\in K^\gamma_{n,m},n,m\in
\mathbb N, \gamma=\omega,\beta]$ onto
$[T^\eta y_{i,j,k}:k\in K^\eta_{i,j}]$ with $\|P^\eta_{i,j} y\|\le \delta
\|y\|$ for
all $y \in
[T^\gamma y^\gamma_{n,m,k}:k\in K^\gamma_{n,m},n,m\in \mathbb N,
\gamma\in\{\omega,\beta\}, (\gamma,n,m)\ne (\eta,i,j)].$
For each $\gamma,n,m,$ let
$R^\gamma_{n,m}$ be the basis
equivalence from $(y^\gamma_{n,m,k})_{k\in K^\gamma_{n,m}}$ to
$(y^\gamma_{n,m,k})_{k=1}^\infty$.
By Lemma \ref{perturbP}, for
each $i,j,\eta$ we may
assume that the projection $P^\eta_{i,j}$ is zero on $[T^\gamma
y^\gamma_{n,m,k}:k\in
K^\gamma_n,n,m\in
\mathbb N, \gamma\in \{\omega,\beta\},(\gamma,n,m)\ne (\eta,i,j)]$.

By the inductive assumption for each $N,M \in \mathbb N$, $V^\beta$
contains a
$\big(\sum_{n=M+N}^\infty Y_n\big)_{c_0}$-tree of order $\beta$ with node
entries drawn from the subspaces $Y^\beta_{n,m},$ $n,m \in \mathbb N.$
Because $(y^{\beta}_{n,m,k})_{k\in K_n}$ is $C$-equivalent to the given
basis of $Y^\beta_{n,m}$ and the stability of $\big(\sum_{n=M+N}^{M+N+l}
Y_n\big)_{c_0}$ under isomorphisms of the $Y_n$, $X$ contains a
$\big(\sum_{n=M+N}^\infty Y_n\big)_{c_0}$-tree of order at least $\beta$
with
constants $C K D$ and $C K$ and node entries from
$X^\beta_{n,m}=[T^\beta y^\beta_{n,m,k}:k\in K^\beta_{n,m}]$, $n,m \in
\mathbb N.$
Similarly, $X$ contains a
$\big(\sum_{n=M}^\infty Y_n\big)_{c_0}$-tree of order at least $\omega$
with
constants $C K D$ and $C K$ and the node entries from
$X^\omega_{n,m}=[T^\omega y^\omega_{n,m,k}:k\in K^\omega_{n,m}]$, $n,m
\in \mathbb N.$
In particular, we can find a branch containing a node of length $N$,
$$\big((X^\omega_{M,m(M)},S^\omega_{M,m(M)}), \dots,
(X^\omega_{M+N-1,m(M+N-1)},S^\omega_{M+N-1,m(M+N-1)})\big)$$
where
$S^\omega_{j,m(j)}=I^\omega_{j,m(j)} R^\omega_{j,m(j)}(T^\omega)^{-1}$
and $I^\omega_{j,m(j)}$ is the
basis isometry from $Y^\omega_{j,m(j)}$ onto $Y_j$. If
\begin{multline*}\big((X^\beta_{M+N,m(M+N)},S^\beta_{M+N,m(M+N)}), \dots,\\
(X^\beta_{M+N+i,m(M+N+i)},S^\beta_{M+N+i,m(M+N+i)})\big)
\end{multline*}
is a node from the
$\big(\sum_{n=M+N}^\infty
Y_n\big)_{c_0}$-tree where $S^\beta_{j,m(j)}$ is defined analogously,
then for all $x_j \in X^\omega_{j,m(j)}$, $M\le j\le M+N-1$, $x_j
\in X^\beta_{j,m(j)}$, $M+N\le j\le M+N+i,$ we have for $G\subseteq \{M,
\ldots, M+N+i\}$
$$\max_{j\in G, \gamma(j)=\omega, \beta}
\|x_j\|\|P^{\gamma(j)}_{j,m(j)}\|^{-1}\le \left\|\sum_{j\in G}
x_j\right\|\le 2C K D\max_{j\in G} \|x_j\|.$$

%\begin{multline*}\max \{\|x_j\|\|P^{\gamma(j)}_{j,m(j)}\|^{-1}:\\
% \gamma(j)=\omega, M\le j
%\le M+N-1,
%\gamma(j)=\beta , M+N\le j \le M+N+i\} \le \\
%\|\sum_{j=M}^{+i} x_j\| \le 2C K \max \{\|x_j\|:M \le j \le M+N+i\}.
%\end{multline*}
Thus for each $N$ we have a tree with order $\beta+N.$ 
We can improve the constants by renorming the closed linear
span
$V^{\beta+\omega}$ of the
spaces
$X^\gamma_{n,m}$, $\gamma\in\{\omega,\beta\},$ $n,m \in \mathbb N,$
by
$$ \|x\|_{\beta+\omega}=\sup\Big\{\|R^\gamma_{m,n}
(T^\gamma_{m,n})^{-1}P^\gamma_{m,n} x\|:\gamma=\omega,\beta; m,n\in
\mathbb N\Big\}\vee
\frac{\|x\|}{2 C K D
}.$$ 
%Let $V_{\beta+\omega}$ be $V^{\beta+\omega}$ with norm $ \|\cdot
%\|_{\beta+\omega}.$
The image of the index $\beta+\omega$ tree in
$V^{\beta+\omega}$ has constants $1$.
By the elastic property $V^{\beta+\omega}$ is
$K$-isomorphic to a subspace of $X$. Thus $X$ contains a
$\big(\sum_{n=M}^\infty Y_n\big)_{c_0}$-tree of order $\beta+\omega$
for all $M \in \mathbb N.$

For the remaining case we have that $\alpha$ is a increasing limit of a
sequence of limit ordinals $(\beta_i).$ This step is similar to the
case of
passing from $V^i$, $i \in \mathbb N$ to $V^\omega$ that was done at
the initial step of the induction.
By the induction hypothesis we have spaces $V^{\beta_i}$ for all $i \in
\mathbb N$ and for each $i$
a countable family of subspaces $Y^{\beta_i}_{n,m}$, $n,m
\in \mathbb N$ such that $Y^{\beta_i}_{n,m}$ is isometric to $Y_n$ and
there is a $\big(\sum_{n=N}^\infty Y_n\big)_{c_0}$-tree in $V^{\beta_i}$
with all
entries in the nodes taken from the subspaces $Y^{\beta_i}_{n,m}$.
For each $i$ let $T_i$ be an isomorphism of $V^{\beta_i}$ into $X$
given by
the elastic property.
Let $(y^{\beta_i}_{n,m,k})_{k=1}^\infty$ be a two-player
complementably sequentially reproducible basis
of $Y^{\beta_i}_{n,m}$ for each $i,n,m \in \mathbb N.$
By interweaving the
two player games for $(y^{\beta_i}_{n,m,k})_{k=1}^\infty$, $i,n,m \in
\mathbb N,$
we can find subsequences $(y^{\beta_i}_{n,m,k})_{k\in K^{\beta_i}_{n,m}}$
of
$(y^{\beta_i}_{n,m,k})_{k=1}^\infty$
for all $i,n,m \in \mathbb N,$ such that for all $i,n,m \in \mathbb N,$
$(T_i y^{\beta_i}_{n,m,k})_{k\in
K^{\beta_i}_{n,m}}$ 
is (equivalent to) a block basis of $(w_s)$ in some order, and for each
$j,l,p$ there is a
projection $P^{\beta_j}_{l,p}$ from $V^{\alpha}:=[(T_i
y^{\beta_i}_{n,m,k})_{k\in
K^{\beta_i}_{n,m}}: i,n,m \in \mathbb N]$
onto
$[T_j y^{\beta_j}_{l,p,k}:k\in K^{\beta_j}_{l,p}]$ with
$\|P^{\beta_j}_{l,p} y\|\le \delta \|y\|$ for
all $y \in
[(T_i y^{\beta_i}_{n,m,k})_{k\in
K^{\beta_i}_{n,m}}: i,n,m \in \mathbb N, (i,n,m) \ne (j,l,p)].$ Let
$R^{\beta_i}_{n,m}$ be the basis
equivalence from $(y^{\beta_i}_{n,m,k})_{k\in K^{\beta_i}_{n,m}}$ to
$(y^{\beta_i}_{n,m,k})_{k=1}^\infty$.
By Lemma \ref{perturbP}, for
each $j,l,p$ we may
assume that the projection $P^{\beta_j}_{l,p}$ is zero on $[(T_i
y^{\beta_i}_{n,m,k})_{k\in
K^{\beta_i}_{n,m}}: i,n,m \in \mathbb N, (i,n,m) \ne (j,l,p)].$

Define a norm on $V^\alpha$
by
$$\|y\|_\alpha= \sup\Big\{\|R^{\beta_i}_{n,m} T_{i}^{-1}P^{\beta_i}_{n,m}
y\|:i,n,m\in \mathbb
N\Big\}\vee
\frac{\|y\|}{K C
},$$
and let $B = \sup_{i,n,m} \|R^{\beta_i}_{n,m}
T_{i}^{-1}P^{\beta_i}_{n,m}\|.$
As before $\frac{\|y\|}{K C } \le \|y\|_\alpha \le B \|y\|$ for all $y \in
V^{\alpha}.$ 
%Let $V_{\alpha}$ be $V^{\alpha}$ with the norm $\|\cdot
%\|_{\alpha}.$
The required properties are easily verified.

This completes the induction step. By Theorem \ref{index_omega1} $X$ contains a
subspace isomorphic to $(\sum Y_n)_{c_0}.$
\end{proof}

Now we deduce  Theorem \ref{mainthm} from Proposition \ref{c0sum}.
\begin{proof}
By Theorem \ref{c0} of Johnson and Odell we know that an elastic space
$X$ with
elastic constant $K$
contains a subspace $Y$ $K$-isomorphic to $c_0.$ $c_0$ is isomorphic to
$C_0(\omega^n)$ for each $n \in \mathbb N,$ and thus $X$ has subspaces
$Y_n$
$K$-isomorphic to  $C_0(\omega^n)$ for each $n \in \mathbb N.$ By
Proposition \ref{reproducible basis} and
Proposition  \ref{c0sum}, $X$ has a subspace isomorphic to
$\big(\sum_{n=1}^\infty C_0(\omega^n)\big)_{c_0}$ which is known
\cite{BP2} to be
isomorphic to $C_0(\omega^{\omega}).$

Inductively, we see that if $X$ contains a subspace $K$-isomorphic to
$C_0(\omega^{\omega^\alpha})$, then because $C_0(\omega^{\omega^\alpha})$
is isomorphic to $C_0(\omega^{\omega^\alpha n})$ for each $n \in \mathbb
N$ \cite{BP2}, by Proposition \ref{reproducible basis} and
Proposition  \ref{c0sum}, $X$ has a subspace $K$-isomorphic to
$\big(\sum_{n=1}^\infty C_0(\omega^{\omega^\alpha n})\big)$ which is
in turn
isomorphic to
$C_0(\omega^{\omega^{\alpha+1}})$. For a limit ordinal $\alpha= \lim
\beta_n$, a similar argument applied to subspaces of $X$
$K$-isomorphic to $C_0(\omega^{\omega^{\beta_n}})$ to obtain a subspace
isomorphic to
$\big(\sum_{n=1}^\infty C_0(\omega^{\omega^{\beta_n}})\big)_{c_0}$. The
latter space is isomorphic to
$C_0(\omega^{\omega^\alpha})$.

Hence $X$ contains subspaces $K$-isomorphic to
$C_0(\omega^{\omega^\alpha})$ for all countable ordinals $\alpha.$
By Theorem \ref{Bourgain}, $X$ contains a subspace isomorphic to
$C[0,1].$
\end{proof}
%\begin{remark} The  $c_0$ sum is important for the proving the main
%theorem, but is not critical in Proposition \ref{c0sum}. With some
%adjustments in the formulas for the norms another unconditional sum could be
%used. We did not pursue this because we needed to be able to construct
%projections onto large subspaces.
%\end{remark}
%\input{as_CK_elastic_ref}

%%%%%%%%%%%%%%%%%

\end{document}